\documentclass[journal]{IEEEtran/IEEEtran}
\ifCLASSINFOpdf
  \usepackage[pdftex]{graphicx}
\else
  \usepackage[dvips]{graphicx}
\fi
\usepackage{tabularx}
\usepackage[width=0.5\textwidth]{caption}
\usepackage[cmex10]{amsmath}
\usepackage{amsthm}
\usepackage[flushleft]{threeparttable}

\usepackage{graphicx, bm, bbm, latexsym, amsfonts, epsfig, verbatim, epsfig, verbatim, enumerate, multirow, graphicx, subcaption, algorithm, algpseudocode, hyperref, url, multicol, color, latexsym, amsfonts, epsfig, verbatim, tikz, cases, algcompatible, diagbox,multicol,fixmath,makecell,booktabs,comment}

\hypersetup{
	colorlinks=true, 
	breaklinks=true,
	urlcolor= blue, 
	linkcolor= blue, 
	citecolor=red, 
	pdftitle={},
	pdfauthor={},
}
\allowdisplaybreaks

\newcommand{\qb}{\bm{q}}
\newcommand{\bb}{\bm{b}}
\newcommand{\cb}{\bm{c}}
\newcommand{\zb}{\bm{z}}
\newcommand{\xb}{\bm{x}}
\newcommand{\yb}{\bm{y}}

\newcommand{\Ab}{\bm{A}}
\newcommand{\Bb}{\bm{B}}
\newcommand{\Cb}{\bm{C}}
\newcommand{\db}{\bm{d}}

\newcommand{\mub}{\bm{\mu}}
\newcommand{\lambdab}{\bm{\lambda}}

\newcommand{\betab}{\bm{\beta}}
\newcommand{\tilxi}{\bm{\tilde{\xi}}}
\newcommand{\tilxiN}{\tilde{\xi}_{\text{\tiny N}}}
\newcommand{\tilxiE}{\tilde{\xi}_{\text{\tiny E}}}

\newcommand{\Eb}{\mathbb{E}}
\newcommand{\Pb}{\mathbb{P}}

\newcommand{\zzg}{z^{\text{\tiny g}}} 
\newcommand{\zza}{z^{\text{\tiny a}}} 

\newcommand{\ffp}{f^{\text{\tiny p}}} 
\newcommand{\ffq}{f^{\text{\tiny q}}} 
\newcommand{\UBgp}{\overline{g}^{\text{\tiny p}}} 
\newcommand{\UBgq}{\overline{g}^{\text{\tiny q}}} 
\newcommand{\LBgp}{\underline{g}^{\text{\tiny p}}} 
\newcommand{\LBgq}{\underline{g}^{\text{\tiny q}}} 

\newcommand{\ddp}{d^{\text{\tiny p}}} 
\newcommand{\ddq}{d^{\text{\tiny q}}} 

\newcommand{\lpp}{l^{\text{\tiny p+}}} 
\newcommand{\lpm}{l^{\text{\tiny p--}}} 
\newcommand{\lqp}{l^{\text{\tiny q+}}} 
\newcommand{\lqm}{l^{\text{\tiny q--}}} 

\newtheorem{theorem}{Theorem}[section]
\newtheorem{remark}[theorem]{Remark}

\newtheorem{proposition}[theorem]{Proposition}
\newtheorem{corollary}[theorem]{Corollary}
\newtheorem{lemma}[theorem]{Lemma}

\begin{document}
\title{Mitigating the Impacts of Uncertain Geomagnetic Disturbances on Electric Grids: A Distributionally Robust Optimization Approach}

\author{Minseok~Ryu \IEEEmembership{Member, IEEE},
        Harsha~Nagarajan \IEEEmembership{Member, IEEE},
        and~Russell~Bent \IEEEmembership{Member, IEEE}
\thanks{M. Ryu is with the Mathematics and Computer Science Division, Argonne National Laboratory, Lemont, IL, USA (Contact: mryu@anl.gov).}
\thanks{H. Nagarajan and R. Bent are with the Applied Mathematics and Plasma Physics group (T-5), Los Alamos National Laboratory, Los Alamos, NM, USA (Contact: \{harsha, rbent\}@lanl.gov).}
}

\maketitle

\begin{abstract}
Severe geomagnetic disturbances (GMDs) increase the magnitude of the electric field on the Earth's surface (E-field) and drive geomagnetically-induced currents (GICs) along the transmission lines in electric grids.
These additional currents can pose severe risks, such as current distortions, transformer saturation and increased reactive power losses, each of which can lead to system unreliability. Several mitigation actions (e.g., changing grid topology) exist that can reduce the harmful GIC effects on the grids. Making such decisions can be challenging, however, because the magnitude and direction of the E-field are uncertain and non-stationary. In this paper, we model uncertain E-fields using the distributionally robust optimization (DRO) approach that determines optimal transmission grid operations such that the worst-case expectation of the system cost is minimized. We also capture the effect of GICs on the nonlinear AC power flow equations. For solution approaches, we develop an \textit{accelerated} column-and-constraint generation (CCG) algorithm by exploiting a special structure of the support set of uncertain parameters representing the E-field. Extensive numerical experiments based on ``epri-21'' and ``uiuc-150'' systems, designed for GMD studies, demonstrate (i) the computational performance of the accelerated CCG algorithm, {\color{black} (ii) the superior performance of distributionally robust grid operations that satisfy nonlinear, nonconvex AC power flow equations and GIC constraints, in comparison with standard stochastic programming-based methods during the out-of-sample testing.}
\end{abstract}


%
\IEEEpeerreviewmaketitle

\section{Introduction}
\label{sec:intro}
Geomagnetic disturbances (GMDs) refer to changes in the geomagnetic field of the Earth that are typically caused by space weather environments.
GMDs induce an electric field (E-field) on the Earth's surface that can be estimated through Earth conductivity models \cite{meqbel2014deep}.
Figure \ref{fig:E-field} shows the time-varying magnitude and direction of the E-field in the United States.

\begin{figure}[!h]
    \centering
    \includegraphics[scale=0.144]{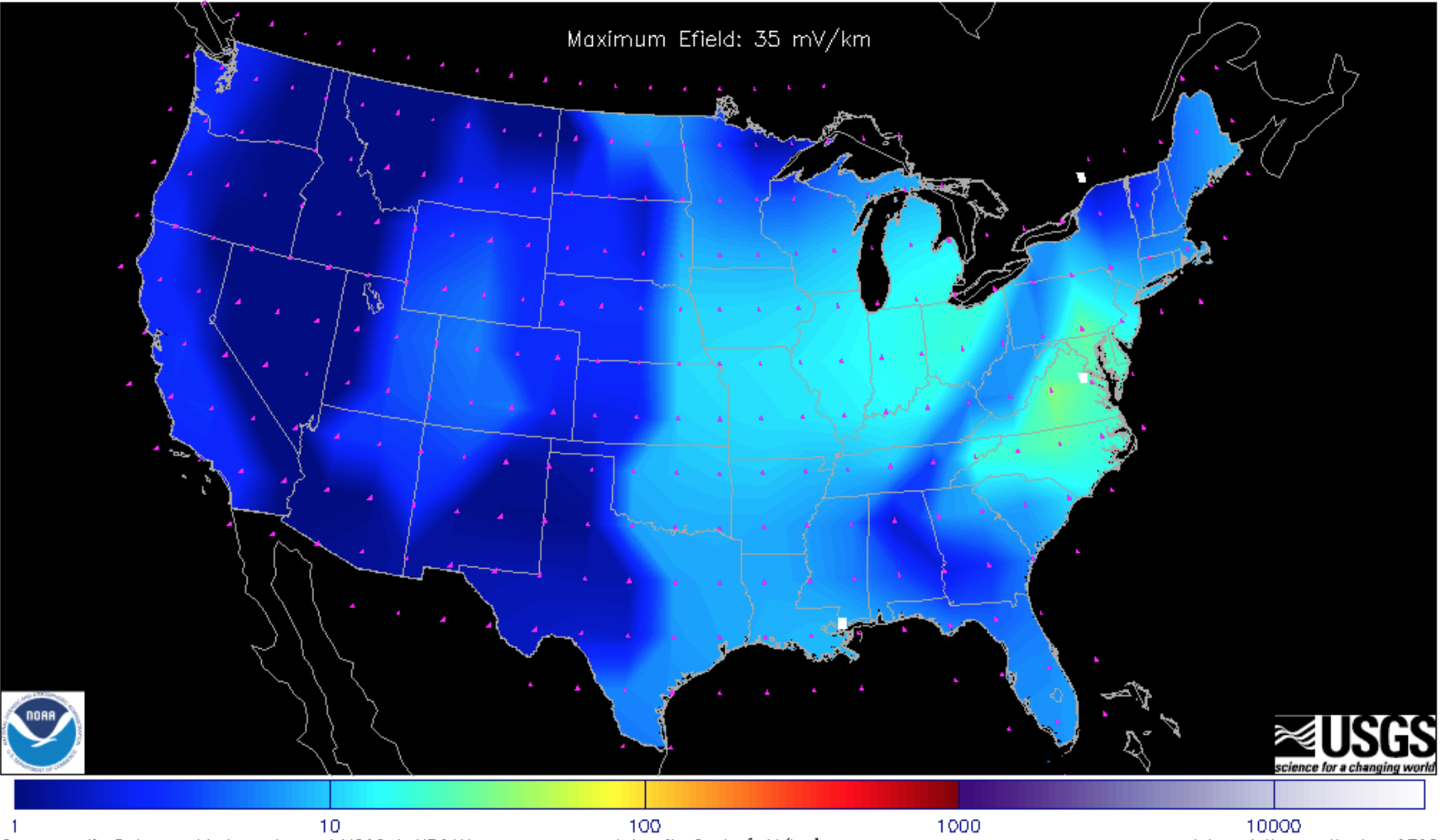}
    \includegraphics[scale=0.138]{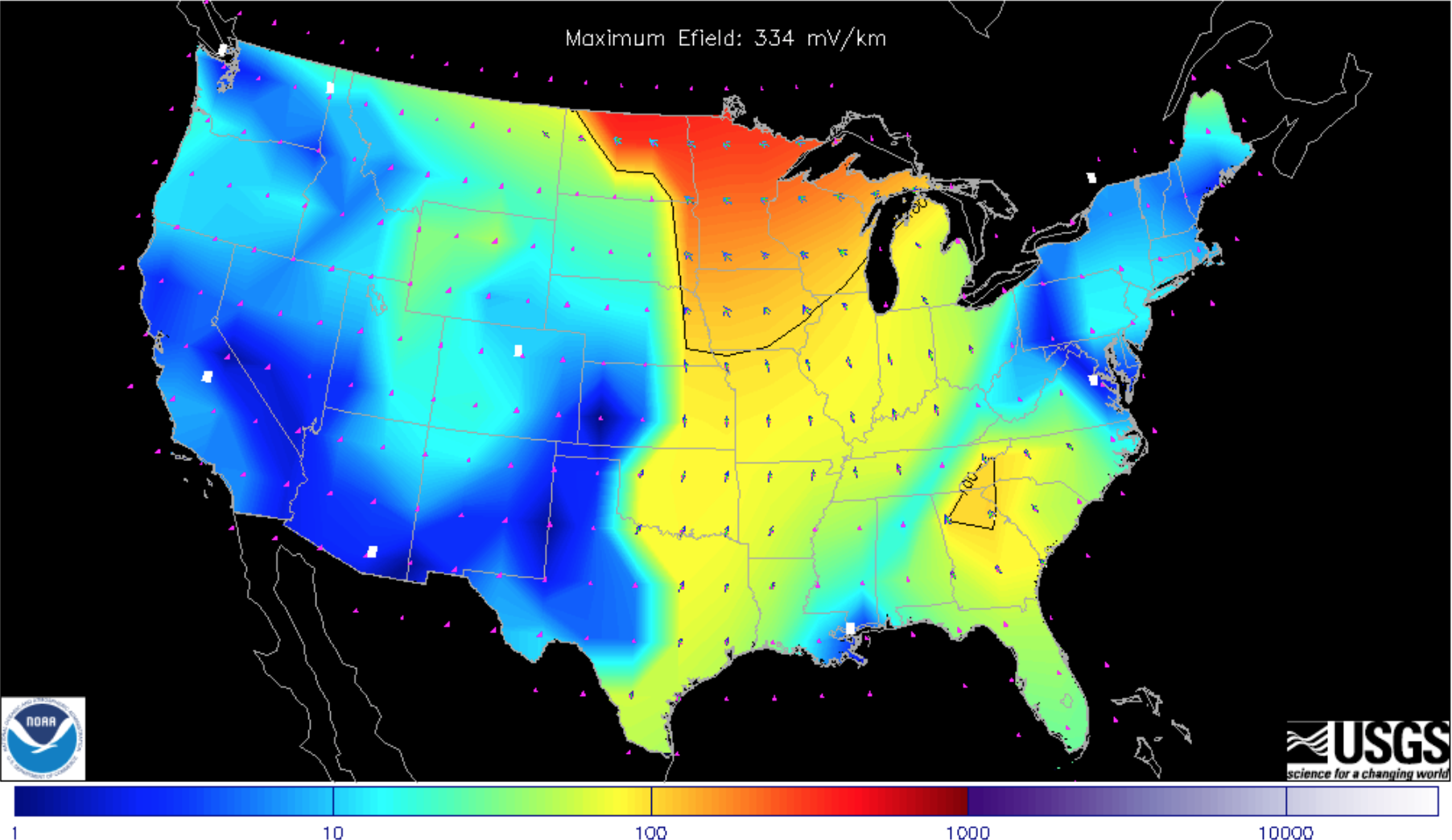}
    \caption{Magnitude and direction of the E-field in the United States \cite{Efield}.}
    \label{fig:E-field}
\end{figure}

When the magnitude of the E-field increases during severe GMDs (e.g., by intense solar storms), it induces significant amounts of geomagnetically-induced currents (GICs) along transmission lines in electric power grids.
These GICs cause partial saturation of power transformers and produce transformer heating, induce distortion of the AC waveform currents, increase reactive power losses, and create unreliability in transmission grids. For example, in 1989, the Hydro-Qu\'ebec power system experienced an $8$ V/km of the E-field, which led to 9 hours of blackout and a net loss of $\$13.2$M \cite{bolduc2002gic}. Similar events occurred in 2003 in Scandinavia and South Africa \cite{gannon2019geomagnetically}.

{\color{black}
Even though GMD events are rare in comparison with other more common uncertain events (e.g., contingencies and renewable energy), it has been identified as a ``high-impact, low-frequency" event causing risk to the power systems as stated in the report from the U.S. Department of Energy \cite{NERC2010}. For this reason, numerous efforts have been made to mitigate the negative impacts of GMDs on the power grid,} such as
(i) improving the Earth conductivity models to estimate the E-field accurately \cite{kelbert2017geoelectric},
(ii) mathematically modeling GICs to calculate the additional reactive power losses due to GICs \cite{north2013application}, and
(iii) establishing mitigation actions that 
protect power grids from severe GICs.
Since this paper belongs to category (iii), we summarize the literature belonging to this category. 

\noindent
\textbf{Literature review}
Several mitigation actions can reduce the harmful GIC effects on the electric transmission grids.
First, installing direct current (DC) blocking devices at regional substations can prevent GICs, which is quasi-DC, from entering the power network through transformer neutrals \cite{bolduc2005development, zhu2014blocking}.
Second, changing the network topology by deactivating a few transmission components (e.g., transmission lines, transformers, and generators) can prevent damage due to GICs. In \cite{lu2017optimal}, the authors proposed an optimal transmission line switching model under GMDs based on alternating current (AC) power flow equations and a set of constraints that captures GIC effects on various types of transformers. 
Utilizing convex relaxations, the researchers formulated the problem as a mixed-integer quadratic convex program.
Later, in \cite{kazerooni2018transformer}, the authors presented heuristic algorithms to mitigate the effect of GICs on transformers by using line switching strategies. \\ 
\indent
Making an optimal \textit{switching decision}, in other words, deciding which transmission components are switched off is important, but challenging because the magnitude and direction of the E-field are \textit{uncertain} when making such a decision, and non-stationary (i.e., changing over time) after making such a decision. Moreover, since GMDs are \textit{low-frequency, high-impact} events, the limited historical data makes it difficult to estimate the underlying distribution of the uncertain event.
One way to address these challenges is to formulate the problem using a distributionally robust optimization (DRO) approach. The following is an example of the DRO model:
\begin{align*}
\min_{x \in \mathcal{X}} \ \Big\{ c^{\text{\tiny T}} x + \sup_{\mathbb{P} \in \mathcal{D}} \mathbb{E}_{\mathbb{P}} [\mathcal{Q}(x,\tilde{\xi})] \Big\},
\end{align*}
where $\mathcal{X}$ is a feasible region of the first-stage problem; $\mathcal{Q}(x,\tilde{\xi})$ is the optimal value of the second-stage problem, which is solved after the uncertain parameter $\tilde{\xi}$ is realized; and
$\mathcal{D}$ is an ambiguity set, that is, a collection of probability distributions of $\tilde{\xi}$ supported on the set $\Xi$.
The DRO model finds a solution $x$ that optimizes the expected outcome $\mathbb{E}_{\mathbb{P}} [\mathcal{Q}(x,\tilde{\xi})]$ over the worst-case probability distribution $\mathbb{P}$ within the ambiguity set $\mathcal{D}$. (A detailed connection of this DRO model with respect to uncertain E-fields will be presented in later sections.) 
The DRO approach has been particularly beneficial in quantifying uncertainty and hence has been adopted by the power system community \cite{zhang2016distributionally, zhao2017distributionally, duan2018distributionally}.
This paper provides a unified DRO-based framework, extends an unpublished technical report~\cite{lu2019distributionally}, and builds on the recent work of \cite{ryu2020pscc}. While the technical report in \cite{lu2019distributionally} observed the value of modeling uncertain GMDs using the DRO approach, it used \textit{linearized} power flow and GIC constraints with nonconvex objectives in the second-stage problem of the two-stage approach, leading to prohibitively long run times on a small test instance. 
In \cite{ryu2020pscc}, we noticed that shedding some loads in advance can mitigate the negative GIC effects, and we proposed a new DRO model that provided switching decisions and total load shedding that minimizes the worst-case expected costs of damage by GICs.
This model consisted of a mixed-integer second-order cone program (MISOCP) in the first-stage and linear second-stage problems, where the MISOCP relaxation of nonconvex AC power flow equations with line switching \cite{kocuk2017new} and linearized GIC constraints were utilized.
The solution approach for the DRO model in \cite{ryu2020pscc} relied on a column-and-constraint generation (CCG) algorithm \cite{zeng2013solving} that was applied on a small-scale instance. This is a standard algorithm for solving two-stage DRO problems, which does not exploit any special structures in the ambiguity sets. Another major drawback of approaches in \cite{lu2019distributionally,ryu2020pscc} is that the optimal solutions involving MISOCP relaxations of the AC power flow may not be feasible to the original nonconvex problem, thus providing only a lower bound on the actual cost of mitigation actions, and making the solutions not of practical interest. However, guaranteeing nonconvex AC feasible solutions in the DRO setting is significantly challenging from both a theoretical and a computational perspective. \\
\indent
In this paper, we address this gap by developing an approach for finding practically viable \textit{transmission grid operations}--- in other words, switching decisions and \textit{feasible} AC and GIC power flow with control of load shedding, that hedge against the worst-case probability distribution of an uncertain E-field.
To this end, we propose a DRO model that is composed of AC power flow equations with switching binary variables at the first stage, and \textit{nonlinear} GIC constraints that calculate the additional reactive power losses due to GICs at the second stage. 
The \textit{contributions} of this paper are as follows: 
{\color{black}
\begin{itemize}
\item[1.] We provide an efficient unifying two-stage DRO-based framework for the problem of mitigating effects of uncertain GMDs on electric grids. To this end, we identify a special structure of the support set of the uncertain E-field that enables us to reformulate the relaxed DRO model in \cite{ryu2020pscc} as a single-stage MISOCP.
\item[2.] Utilizing the MISOCP reformulation, we develop an \textit{accelerated} CCG algorithm for the relaxed DRO model to scale on medium-scale benchmark instances. 
\item[3.] Given switching decisions, we utilize the accelerated CCG algorithm to solve the DRO model proposed in this paper via a sampling-based approach to obtain AC and GIC feasible power flows. We further establish a clear value proposition of the proposed DRO model in comparison with classical stochastic programming, where out-of-sample methods suggest the superior performance of DRO method on test networks.  
\end{itemize}
}
{\color{black} Although we acknowledge the importance of modeling other sources of uncertainty such as contingencies \cite{babaei2020distributionally} and renewable energy \cite{xiong2016distributionally} in the DRO setting, in this paper we only focus on uncertainties due to GMD events in order to keep the modeling and computational aspects of the problem tractable.}

In the remainder of this paper, we describe the DRO models in Section \ref{sec:math_form}, their corresponding solution approaches in Section \ref{sec:methodoloties}, and numerical experiments in Section \ref{sec:numerical}. We summarize our conclusions in Section \ref{sec:conclusion}.\\
\noindent
\textbf{Notation:} 
We use boldfaces notation to denote vectors.
We define $[m] :=\{1,\ldots,m\}$ for any $m \in \mathbb{N}$, and $[x]_+=\max\{x, 0\}$.

\section{Mathematical Formulation}
\label{sec:math_form}
In this section we propose a DRO model that is composed of two stages:
(i) the first-stage problem, built on AC power flow equations with binary variables, is solved under uncertain magnitude and direction of the E-field, and
(ii) the second-stage problem, built on the nonlinear GIC constraints, is solved given the first-stage solution and the realized magnitude and direction of the E-field.

{\color{black}Since GICs are quasi-DC flows, the DC flow analysis is used to calculate GICs as in \cite{albertson1981load, overbye2012integration}, which are described in Sections \ref{subsec:ACDC} and \ref{subsec:GIC}.}
In Section \ref{subsec:Ambiguity} we describe an ambiguity set, which is a set of probability distributions of the uncertain E-field. 
Then we describe the proposed DRO model in Section \ref{subsec:2stage}.
Table \ref{tab-definition} shows nomenclature used in this paper.

\begin{table}[h!]
    \footnotesize
    \centering
    \caption{Nomenclature}
    \label{tab-definition}    
    \begin{tabular}{l||l}
		\hline
		\multicolumn{2}{c}{\textbf{Sets and parameters}} \\ \hline
		$\mathcal{G},\ \mathcal{N}, \ \mathcal{E}$ & set of generators, buses, and lines in AC network \\
		$\mathcal{E}^{\tau} \subseteq \mathcal{E} $ & set of transformers \\
		$\mathcal{E}_i \subseteq \mathcal{E}$ & set of lines connected to $i \in \mathcal{N}$ \\
		$c^{\text{\tiny F0}}_k$ & fixed cost when switching on $k \in \mathcal{G}$\\
		$c^{\text{\tiny F1}}_k, c^{\text{\tiny F2}}_k$ & fuel cost coefficients of power generation of $k \in \mathcal{G}$ \\
		$\LBgp_k, \UBgp_k$ & bounds on the real power generation of $k \in \mathcal{G}$ \\
		$\LBgq_k, \UBgq_k$ & bounds on the reactive power generation of $k \in \mathcal{G}$ \\
		$\kappa^{\text{\tiny l}}$ & unit penalty cost for power unbalance at $i \in \mathcal{N}$ \\
		$\ddp_i, \ddq_i$ & real and reactive power demand at $i \in \mathcal{N}$ \\
		$\underline{v}_i, \overline{v}_i$ & voltage limits at $i \in \mathcal{N}$ \\
		$g^{\text{\tiny s}}_i, b^{\text{\tiny s}}_i$ & shunt conductance and susceptance at $i \in \mathcal{N}$ \\
		$g_e, b_e$ & conductance, susceptance of $e \in \mathcal{E}$ \\
        $b_e^{\text{\tiny c}}$ & line charging susceptance of  $e \in \mathcal{E}$ \\		
		$\overline{s}_e$ & apparent power limit of line $e \in \mathcal{E}$\\
		$\underline{\theta}_{ij}, \overline{\theta}_{ij}$ & bounds on the phase angle difference $\theta_i-\theta_j$ at $(i,j) \in \mathcal{E}$ \\
		$\alpha_{ij}$ & tap ratio of $e_{ij} \in \mathcal{E}$\\
		$k_e$ & loss factor of transformer $e \in \mathcal{E}^{\tau}$\\
		$\overline{I}^{\text{\tiny eff}}_e$ & upper limit of the effective GICs on $e \in \mathcal{E}^{\tau}$\\
		\hline 
		$\mathcal{N}^{\text{\tiny d}}, \mathcal{E}^{\text{\tiny d}}$ & set of nodes and arcs in DC network \\
		$\mathcal{E}^{\text{\tiny d}-}_m, \mathcal{E}^{\text{\tiny d}+}_m$ & set of incoming and outgoing arcs connected to $m \in \mathcal{N}^{\text{\tiny d}}$ \\
		$\gamma_{\ell}$ & conductance of $\ell \in \mathcal{E}^{\text{\tiny d}}$ \\
		$a_m$ & inverse of ground resistance at $m \in \mathcal{N}^{\text{\tiny d}}$  \\
		$\overline{v}^{\text{\tiny d}}$ & bound on the GIC-induced voltage magnitude \\
		$\tilde{\xi}_{\ell} $ & (random) GIC-induced voltage sources on $\ell \in \mathcal{E}^{\text{\tiny d}}$ \\ 
		\hline
		\multicolumn{2}{c}{\textbf{Variables}} \\ \hline
		$\zza_e \in \mathbb{B}$ & $\zza_e=1$ if $e \in \mathcal{E}$ is turned on, and $\zza_e=0$ otherwise  \\
		$\zzg_k \in \mathbb{B}$ & $\zzg_k=1$ if $k \in \mathcal{G}$ is turned on, and $\zzg_k=0$ otherwise  \\		
		$v_i$ & voltage magnitude at $i \in \mathcal{N}$ \\
		$\theta_{ij}$ & phase angle difference ($\theta_{i}-\theta_{j}$) on $e_{ij} \in \mathcal{E}$ \\
		$\ffp_k, \ \ffq_k$ & real and reactive power generated by $k \in \mathcal{G}$\\
		$p_{ei}, \ p_{ej}$ & real power flow on $e_{ij} \in \mathcal{E}$ \textit{from node} $i$ and \textit{to node} $j$ \\
		$q_{ei}, \ q_{ej}$ & reactive power flow on $e_{ij} \in \mathcal{E}$ \textit{from node} $i$ and \textit{to node} $j$ \\		
		$ \lpp_i, \lqp_i$ & real and reactive load shedding at $i \in \mathcal{N}$ \\
		$ \lpm_i, \lqm_i$ & real and reactive power loss at $i \in \mathcal{N}$ \\		
		$d^{\text{\tiny qloss}}_i$ & additional reactive power loss due to GICs at $i \in \mathcal{N}$ \\		
		\hline 
		$I^{\text{\tiny d}}_{\ell}$ & GICs that flow on $\ell \in \mathcal{E}^{\text{\tiny d}}$ \\
        $I^{\text{\tiny eff}}_e$ & effective GICs on $e \in \mathcal{E}^{\tau}$ \\				
        $v^{\text{\tiny d}}_m$ & GIC-induced voltage magnitude at $m \in \mathcal{N}^{\text{\tiny d}}$  \\		
		\hline
	\end{tabular}
\end{table}
\subsection{Topological mapping of AC and DC power network}
\label{subsec:ACDC}
We use the standard representation of the AC power network, represented by a graph $(\mathcal{N}, \mathcal{E})$, where $\mathcal{N}$ is the set of buses and $\mathcal{E}$ is the set including both transmission lines and transformers ($\mathcal{E}^{\tau} \subset \mathcal{E}$). 
For evaluating GICs (quasi-DC flow), we construct a corresponding DC power network $(\mathcal{N}^{\text{\tiny d}}, \mathcal{E}^{\text{\tiny d}})$, where $\mathcal{N}^{\text{\tiny d}}$ includes buses in $\mathcal{N}$ and additional nodes that model the neutral grounding points of transformers.
The set $\mathcal{E}^{\text{\tiny d}}$ includes transmission lines in $\mathcal{E}$ and additional lines between the end points of transformers and their neutrals. In this paper, we consider three transformer configurations, Gwye-Gwye, GWye-GWye Auto, and GWye-Delta GSU, whose descriptions can be found in \cite{lu2017optimal,ryu2020pscc}. 

Figure \ref{fig:ACDC} shows the topological mapping between the AC and DC network representations modeled in this paper. In the AC network, $\mathcal{N} = \{3,4,15,16,17,18\}$, $\mathcal{E}=\{8,12,14,18,28,30\}$, and $\mathcal{E}^{\tau}=\{18,28,30\}$. The DC network $(\mathcal{N}^{\text{\tiny d}}, \mathcal{E}^{\text{\tiny d}})$ is constructed by modeling different types of transformers and their connection to grounding neutrals, $\{G_1,G_2,G_3\}$. The set $\mathcal{N}$ is relabeled with $\{3^d,4^d,15^d,16^d,17^d,18^d\}$ in the DC network. 
To link the two networks, we define $E$ and $E^{-1}$ that map $\ell \in \mathcal{E}^{\text{\tiny d}}$ to an edge $e \in \mathcal{E}$ and vice versa; that is, if $E_{\ell}=e$, then $E^{-1}_e = \{ \ell \in \mathcal{E}^{\text{\tiny d}} : E_{\ell} = e  \}$. 
For example, line $17$ in the DC network maps to transformer line $18$ in the AC network; thus $E_{17}=18$. Transformer line $18$ maps to lines $16$ and $17$ in the DC network; thus $E^{-1}_{18} = \{16, 17\}$. {\color{black} Finally, in the DC representation of the network, transmission lines with series capacitive compensation are omitted as series capacitors can block the flow of induced GICs.}
\begin{figure}[!h]
    \centering
    \includegraphics[scale=0.45]{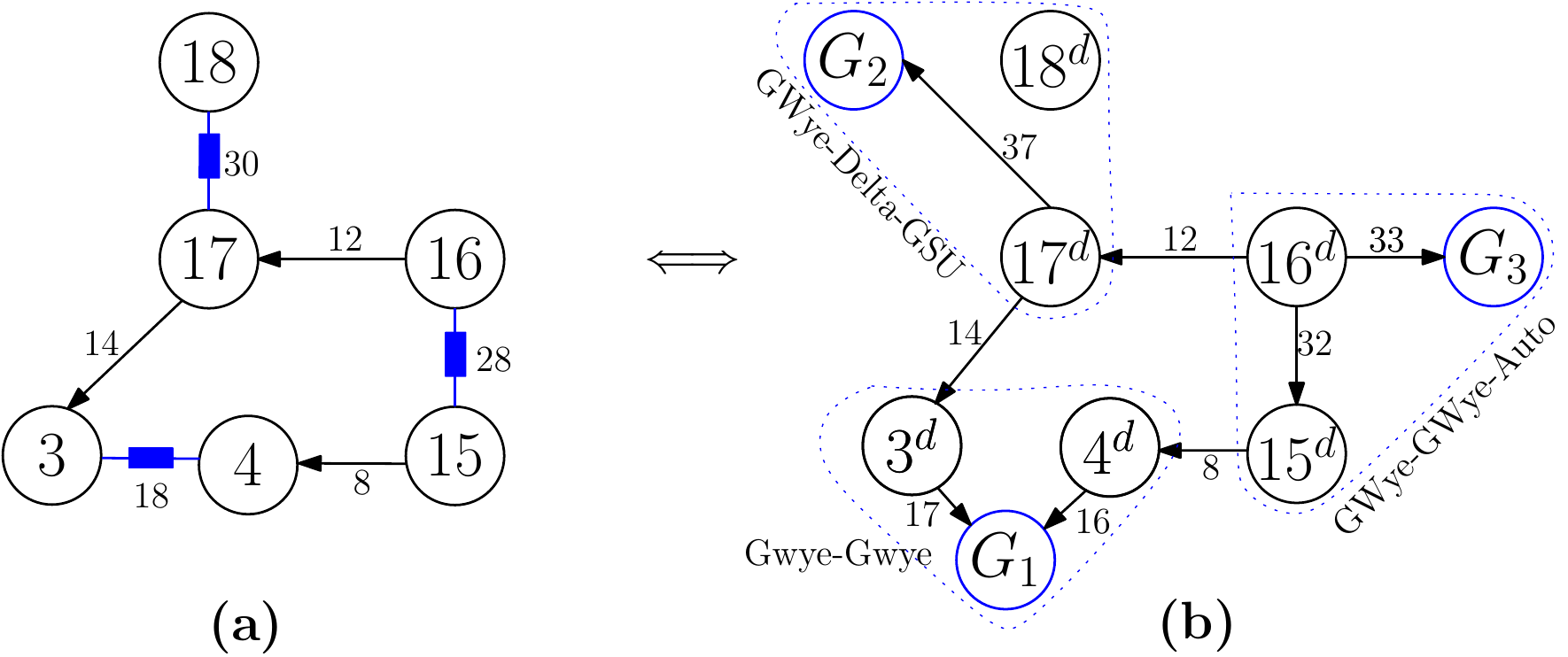}
    \caption{(a) Example of an AC power network (with three transformer types). (b) Equivalent DC network map.}
    \label{fig:ACDC}
    \vspace{-0.5cm}
\end{figure}

\subsection{Modeling of GICs and reactive power losses} \label{subsec:GIC}
{\color{black} To model the GICs, we assume that the main phase of the geomagnetic storm is produced by a ring current in the radiation belt above the Earth’s surface and produces large-scale disturbances that can be reasonably assumed as a spatially uniform E-field across the electric grid \cite{boteler1998modelling}. This is also consistent with the modeling assumptions of \cite{overbye2012integration, horton2012test}.}

{\color{black} Since E-field measurements are usually recorded in the eastward (horizontal) and northward (vertical) directions \cite{team2016benchmark}}, 
let $\tilxiE$, $L^{\text{\tiny E}}_{\ell}$ and $\tilxiN, L^{\text{\tiny N}}_{\ell}$ be E-fields (V/km) and lengths (km) of transmission lines in the eastward and northward direction, respectively. {\color{black} The GIC-induced voltage source, $\tilde{\xi}_{\ell}$, acting on a transmission line, $\ell$}, is given by 
\begin{align*}
\tilde{\xi}_{\ell} = 
\begin{cases}
\tilxiE L^{\text{\tiny E}}_{\ell} + \tilxiN L^{\text{\tiny N}}_{\ell}, \ \forall \ell \in \mathcal{E}^{\text{\tiny d}} : E_{\ell} \in \mathcal{E} \setminus \mathcal{E}^{\tau}, \\ 
0, \ \forall \ell \in \mathcal{E}^{\text{\tiny d}} : E_{\ell} \in \mathcal{E}^{\tau}.
\end{cases}
\end{align*}
Then, the GIC flow on a line between two nodes $m$ and $n$ is 
\begin{align*}
I^{\text{\tiny d}}_{\ell} = \gamma_{\ell} (v^{\text{\tiny d}}_m - v^{\text{\tiny d}}_n + \tilde{\xi}_{\ell}), \ \forall \ell_{mn} \in \mathcal{E}^{\text{\tiny d}}.
\end{align*}
Based on the GIC in the DC network, the effective GIC of transformers in the AC network is given by:
\begin{align*}
& I^{\text{\tiny eff}}_e = | \Theta(I^{\text{\tiny d}}_{\ell}, \ \forall \ell \in E^{-1}_e )|, \ \forall e \in \mathcal{E}^{\tau},
\end{align*}
where $\Theta(I^{\text{\tiny d}}_{\ell}, \ \forall \ell \in E^{-1}_e )$ is a linear function of GICs ($I^{\text{\tiny d}}_{\ell}$) in various transformer types. The three types of transformer winding configurations and the associated GIC calculation equations are presented in Table \ref{tab-effGIC}.
Note that $(N_h, N_l,N_s, N_c)$ are parameters which indicate the number of turns in the high-side/low-side/series/common winding, respectively. {\color{black} It is noteworthy to mention that most test networks in the literature of GMD studies neglect GSU-type transformers \cite{lu2017optimal}. However, the GSU transformers and the neutral leg ground points provided by them are critical when developing methods to mitigate the impact of GICs on electric grids. Thus, in this paper, we incorporate GSUs, modeled as delta-gywe transformers.} 
\begin{table}[h!]
    \small
    \centering
    \caption{Effective GICs for various transformer types. {\color{black}Here, gwye stands for grounded-wye and GSU for generator step-up transformer which connects the output terminals of generators to the electric grid.}}
    \label{tab-effGIC}    
    \begin{tabular}{l|l|l}
    \toprule	
    Type of transformer $(e)$ & $E^{-1}_e$ 	& $\Theta(I^{\text{\tiny d}}_{\ell}, \ \forall \ell \in E^{-1}_e )$ \\ 
    \hline
    Gwye-Gwye		& $\{h, l\}$  	& $\Theta(I^{\text{\tiny d}}_{h}, I^{\text{\tiny d}}_{l}) = \frac{ N_h I^{\text{\tiny d}}_{h} + N_l I^{\text{\tiny d}}_{l}}{N_h}$ \\
    GWye-GWye-Auto	& $\{s, c\}$  	& $\Theta(I^{\text{\tiny d}}_{s}, I^{\text{\tiny d}}_{c}) = \frac{N_s I^{\text{\tiny d}}_{s} + N_c I^{\text{\tiny d}}_{c}}{N_s+N_c}$	 \\
    GWye-Delta-GSU	& $\{h\}$ &  $\Theta(I^{\text{\tiny d}}_{h}) = I^{\text{\tiny d}}_{h} $ \\ 
    \bottomrule
    \end{tabular}
\end{table}

Given that $\mathcal{E}^{\tau}_i$ is a set of transformers connected to node $i$, the induced reactive power loss \cite{zhu2014blocking} due to GICs at node $i$ in the AC network is given by
\begin{align*}
d^{\text{\tiny qloss}}_i = \sum_{e \in \mathcal{E}^{\tau}_i} k_e v_i I^{\text{\tiny eff}}_e, \ \forall i \in \mathcal{N}.
\end{align*}

\subsection{Ambiguity set description}
\label{subsec:Ambiguity}
{\color{black} Although the E-field is usually recorded in eastward ($\tilxiE$) and northward ($\tilxiN$) directions as a function of time \cite{team2016benchmark}, the actual magnitude and direction of the E-field are uncertain (in the first-stage). Moreover, the correlation between $\tilxiE$ and $\tilxiN$ is unknown.} Generally, there is not sufficient historical information to construct reasonable probability distributions of these uncertain parameters \cite{woodroffe2016latitudinal}. 
Moreover, they are non-stationary during the GMD event.
For these reasons, introducing a specific distribution into an optimization model may lead to biased decisions.

Instead, we construct a set of distributions, or ambiguity set, and find a solution that hedges against the worst-case distribution in the ambiguity set.
Specifically, we construct a mean-support ambiguity set $\mathcal{D}$, which is a set of distributions that have the same mean values and support set:
\begin{align*}
\mathcal{D} := \{ \mathbb{P} \ : \  \mathbb{E}_{\mathbb{P}_{\text{\tiny E}}} [ \tilxiE ] = \mu_{\text{\tiny E}}, \ \  \mathbb{E}_{\mathbb{P}_{\text{\tiny N}}} [ \tilxiN ] = \mu_{\text{\tiny N}}, \ \ \mathbb{P}\{ \tilxi \in \Xi \}=1  \}, 
\end{align*}
where $(\mathbb{P}_{\text{\tiny E}}, \mathbb{P}_{\text{\tiny N}})$ and $(\mu_{\text{\tiny E}}, \mu_{\text{\tiny N}})$ are the marginal distributions and mean values of $(\tilxiE, \tilxiN)$, respectively, and 
$\mu_{\text{\tiny E}} = M \cos \delta^{\mu}$ and $\mu_{\text{\tiny N}} = M \sin \delta^{\mu}$, where $M$ and $\delta^{\mu}$ represent the mean values of random magnitude $\tilde{M}$ and direction $\tilde{\delta}$ of the E-field, respectively.

There are several ways of constructing the support set $\Xi$.
First, one can set bounds on each uncertain parameters, for example, $\tilde{M} \in [0,R]$ and $\tilde{\delta} \in [0^{\circ}, 180^{\circ}]$ in \cite{lu2019distributionally}, where $R$ is set as a maximum magnitude of the E-field.
This type of support set is shown as a half-circle with radius $R$ in Figure \ref{fig:SupportSets}. {\color{black}Note that this type of support set may be considered conservative without correlations on $\tilxiE$ and $\tilxiN$}.
Second, one can construct an approximation polytope with $N$ extreme points as shown in Figure \ref{fig:SupportSets}.
As $N$ increases, this polyhedral support set eventually converges to the half-circle. 
In our previous work \cite{ryu2020pscc}, we numerically showed that a polytope with $5$ extreme points provides solutions that are close to solutions obtained by the half-circle support set. 
Hence, in this paper we focus on polyhedral support sets, namely, $\Xi = \text{conv} \{ (\hat{\xi}_{\text{\tiny E}}^{\ell}, \hat{\xi}_{\text{\tiny N}}^{\ell}), \ \forall \ell \in [N] \}$, where $(\hat{\xi}_{\text{\tiny E}}^{\ell}, \hat{\xi}_{\text{\tiny N}}^{\ell})$ represents an extreme point of $\Xi$, (i.e., red dots in Figure \ref{fig:SupportSets}).
\begin{figure}[h!]
	\centering
	\includegraphics[scale=0.35]{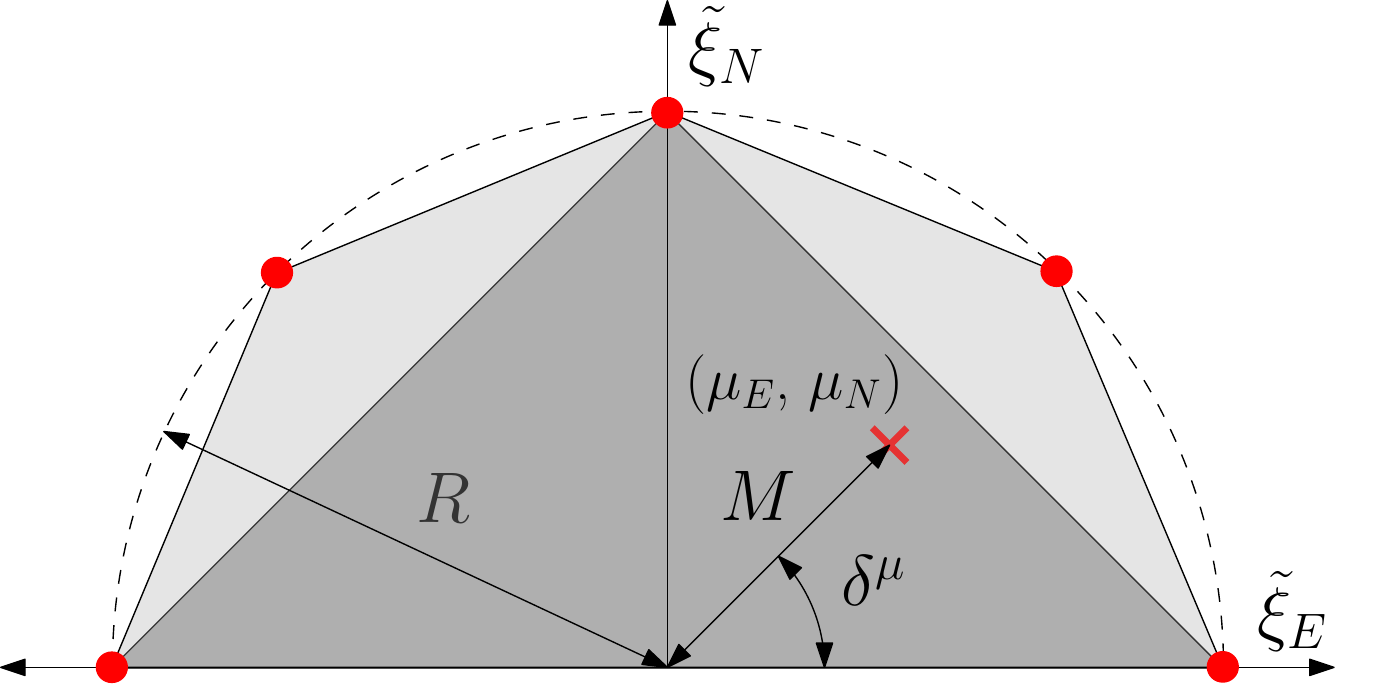}
	\caption{Polyhedral support sets and mean of the E-field $(\tilxiE, \tilxiN)$.}
	\label{fig:SupportSets}
\end{figure}

\subsection{DRO model} \label{subsec:2stage}
In this section, we describe the two-stage DRO model: 
\begin{subequations} 
    \label{DRO:stg1}
	\small{
	\begin{align} 
	\min \ &  \sum_{k \in \mathcal{G}} ( c^{\text{\tiny F0}}_k \zzg_k + c^{\text{\tiny F1}}_k \ffp_k +  c^{\text{\tiny F2}}_k (\ffp_k)^2 ) + \sum_{i \in \mathcal{N}} \kappa^{\text{\tiny l}} ( \lpp_i + \lpm_i + \lqp_i + \lqm_i  ) \nonumber \\ 
	 &\hspace{102pt} + \sup_{\Pb \in \mathcal{D}} \Eb_{\Pb}[ \mathcal{V}(\boldsymbol{{\zza}, v, {d^{\text{\tiny qloss}}}, {\tilde{\xi}}}] \label{DRO:stg1_obj}  \\ 
	\mbox{s.t.} \ 
	& \forall i \in \mathcal{N}: \nonumber \\
	& \ \ \sum_{e \in \mathcal{E}_i } p_{ei} = \sum_{k \in \mathcal{G}_i } \ffp_k - \ddp_i + \lpp_i - \lpm_i - g^{\text{\tiny s}}_i w_i, \label{DRO:stg1_1} \\
	& \ \ \sum_{e \in \mathcal{E}_i } q_{ei} = \sum_{k \in \mathcal{G}_i } \ffq_k - \ddq_i  + \lqp_i - \lqm_i + b^{\text{\tiny s}}_i w_i - d^{\text{\tiny qloss}}_i, \label{DRO:stg1_2} \\
	& \forall e_{ij} \in \mathcal{E}: \nonumber \\
	& \ \ p_{ei} = \zza_e \bigg\{  \frac{1}{\alpha_{ij}^2} g_e w_i - \frac{1}{\alpha_{ij}} (g_e w^{\text{\tiny c}}_e + b_e w^{\text{\tiny s}}_e) \bigg\}, \label{DRO:stg1_3}  \\
	& \ \ p_{ej} =  \zza_e \bigg\{ g_e w_j - \frac{1}{\alpha_{ij}}  (g_e w^{\text{\tiny c}}_e - b_e w^{\text{\tiny s}}_e  ) \bigg\}, \label{DRO:stg1_4} \\
	& \ \ q_{ei} = \zza_e \bigg\{  -\frac{1}{\alpha_{ij}^2} (b_e + \frac{b_e^{\text{\tiny c}}}{2}) w_i + \frac{1}{\alpha_{ij}} (b_e w^{\text{\tiny c}}_e - g_e w^{\text{\tiny s}}_e ) \bigg\}, \label{DRO:stg1_5} \\
	& \ \ q_{ej} =  \zza_e \bigg\{ -(b_e + \frac{b_e^{\text{\tiny c}}}{2}) w_j + \frac{1}{\alpha_{ij}} (b_e w^{\text{\tiny c}}_e + g_e w^{\text{\tiny s}}_e ) \bigg\}, \label{DRO:stg1_6} \\
	& w_i = v_i^2, \ \forall i \in \mathcal{N}, \label{DRO:stg1_7} \\
	&  w^{\text{\tiny c}}_e = v_i v_j \cos (\theta_{ij}),\quad w^{\text{\tiny s}}_e = v_i v_j \sin (\theta_{ij}), \ \forall e_{ij} \in \mathcal{E}, \label{DRO:stg1_9} \\
	& p_{ei}^2 + q_{ei}^2 \leq \zza_e (\overline{s}_e)^2 , \ \ p_{ej}^2 + q_{ej}^2 \leq \zza_e (\overline{s}_e)^2, \ \forall e_{ij} \in \mathcal{E}, \label{DRO:stg1_10} \\
	& v_i \in [\underline{v}_i, \overline{v}_i], \ \forall i \in \mathcal{N},\quad \theta_{ij} \in [\underline{\theta}_{ij}, \overline{\theta}_{ij}], \ \forall e_{ij} \in \mathcal{E}, \label{DRO:stg1_11} \\
	& \LBgp_k \zzg_k \leq \ffp_k \leq \UBgp_k \zzg_k, \quad \LBgq_k \zzg_k \leq \ffq_k \leq \UBgq_k \zzg_k, \ \forall k \in \mathcal{G}, \label{DRO:stg1_13} \\
    & \sum_{e \in \mathcal{E}_k} \zza_e \geq \zzg_k, \  \forall k \in \mathcal{G}, \label{DRO:stg1_14} \\
    & \lpp_i, \lpm_i, \lqp_i, \lqm_i \geq 0, \ \forall i \in \mathcal{N},  \label{DRO:stg1_15} \\	
	& \zza_e \in \{0,1\}, \ \forall e \in \mathcal{E}, \ \ \zzg_k \in \{0,1\}, \ \forall k \in \mathcal{G}. \label{DRO:stg1_16}
	\end{align}}
\end{subequations}
where, given the values of $\boldsymbol{{\zza}, v, {d^{\text{\tiny qloss}}}}$ from the first-stage problem \eqref{DRO:stg1} and the realization of the uncertain E-field $\boldsymbol{\tilde{\xi}}$,  $\mathcal{V}(\boldsymbol{{\zza}, v, {d^{\text{\tiny qloss}}}, {\tilde{\xi}}})$ is the optimal value of the following second-stage problem that evaluates GICs in the DC network:

\begin{subequations}
\small{
\begin{align}
	\min \ & \sum_{i \in \mathcal{N}} \kappa^{\text{\tiny s}} s_i \label{DRO:stg2_obj} \\
	\mbox{s.t.} \ 
    & \sum_{\ell \in \mathcal{E}^{d-}_m} I^{\text{\tiny d}}_{\ell} - \sum_{\ell \in \mathcal{E}^{d+}_m} I^{\text{\tiny d}}_{\ell} = a_m v^{\text{\tiny d}}_m, \ \forall m \in \mathcal{N}^{\text{\tiny d}}, \label{DRO:stg2_1}\\	
	& I^{\text{\tiny d}}_{\ell} =  \zza_{E_{\ell}} \bigg\{ \gamma_{\ell} (v^{\text{\tiny d}}_m - v^{\text{\tiny d}}_n + \tilde{\xi}_{\ell}) \bigg\}, \ \forall \ell_{mn} \in \mathcal{E}^{\text{\tiny d}}, \label{DRO:stg2_2} \\
	& I^{\text{\tiny eff}}_e = \lvert \Theta(I^{\text{\tiny d}}_{\ell}, \ \forall \ell \in E^{-1}_e ) \rvert, \ \forall e \in \mathcal{E}^{\tau}, \label{DRO:stg2_3} \\
	& s_i \geq \bigg[\sum_{e \in \mathcal{E}^{\tau}_i} k_e v_i I^{\text{\tiny eff}}_e - d^{\text{\tiny qloss}}_i \bigg]_+, \ \forall i \in \mathcal{N}. \label{DRO:stg2_4} 
\end{align}}
\label{DRO:stg2}
\end{subequations}
Constraints \eqref{DRO:stg1_1}--\eqref{DRO:stg1_9} model the AC power flow. These constraints include switching decision ($\boldsymbol{\zza, \zzg}$). 
Constraints \eqref{DRO:stg1_1} and \eqref{DRO:stg1_2} model real and reactive power balance constraints, including additional reactive power loss due to GICs ($d_i^{\text{\tiny qloss}}$). Constraints \eqref{DRO:stg1_3}--\eqref{DRO:stg1_9} model Ohm's law. 
Constraints \eqref{DRO:stg1_10} ensure that the apparent power flow does not exceed its limit.
Constraints \eqref{DRO:stg1_11}--\eqref{DRO:stg1_13} ensure that  voltage magnitudes, phase angle differences, and the amounts of real/reactive power generation are within their bounds.
Constraints \eqref{DRO:stg1_14} ensure that a generator is turned off when all the lines and transformers connected to that generator are off. 

In the second-stage problem, 
{\color{black} the objective function \eqref{DRO:stg2_obj} represents the damage due to GICs via the incurred reactive power losses, penalized by the unit cost $\kappa^{\text{\tiny s}}$.
}
Constraints \eqref{DRO:stg2_1}--\eqref{DRO:stg2_2} calculate valid GICs in the DC network for lines that are switched on. Otherwise, they are deactivated by constraints \eqref{DRO:stg2_2}.
Constraints \eqref{DRO:stg2_1} model the nodal balance equation for GICs in the DC network. Note that, in \eqref{DRO:stg2_1}, $a_m=0$ when $m$ is not a grounded neutral node. 
Constraints \eqref{DRO:stg2_3} calculate the effective GICs for each type of transformer (see Table \ref{tab-effGIC}), which in turn are used to calculate the additional reactive power losses due to GICs $\sum_{e \in \mathcal{E}^{\tau}_i} k_e v_i I^{\text{\tiny eff}}_e$ for every bus in the AC network, as shown in constraints \eqref{DRO:stg2_4}. 

\noindent

\textbf{Interpretation:} 
Given the first-stage solution after minimizing the operation and load-shedding costs, including switching decision $\boldsymbol{\zza}$, voltage magnitude $\boldsymbol{v}$, and the (allowable) reactive power loss, $\boldsymbol{d^{\text{\tiny qloss}}}$ 
and the realization of the uncertain E-field $\tilxi$, the second-stage problem calculates the effective GICs and the (actual) additional reactive power losses, given by $\sum_{e \in \mathcal{E}^{\tau}_i} k_e v_i I^{\text{\tiny eff}}_e$. If this exceeds the allowable $\boldsymbol{d^{\text{\tiny qloss}}}$ from the first stage, then appropriate switching decisions are updated in the AC network to mitigate the negative effects on the transformers. Implicitly, this formulation assumes that reactive losses smaller than $\boldsymbol{d^{\text{\tiny qloss}}}$ will not cause voltage problems (i.e., a high voltage).

\section{Solution Methodologies}
\label{sec:methodoloties}

{\color{black}
For ease of exposition, we express the DRO model \eqref{DRO:stg1} as 

\vspace{-2mm}
\begin{small}
	\begin{align}
	\min_{(\zb, \yb) \in \mathcal{A}} \ & \Big\{ \qb^{\text{\tiny T}} \zb + \bb^{\text{\tiny T}} \yb  + \sup_{\mathbb{P} \in \mathcal{D}} \mathbb{E}_{\mathbb{P}}[ \mathcal{V} (\zb, \yb, \tilxi) ] \Big\},
	\label{DRO-Simple}
	\end{align}
\end{small}	
where $\zb, \yb$ are the first-stage binary and continuous vectors, respectively, $\mathcal{A}$ is a feasible region of the first-stage problem \eqref{DRO:stg1}, and $\mathcal{V} (\zb, \yb, \tilxi)$ is the optimal value of the second-stage problem \eqref{DRO:stg2}.

Solving the DRO model \eqref{DRO-Simple} is challenging because of the existence of binary variables $\bm{z}$ representing switching decisions and nonlinear, nonconvex constraints (e.g., AC power flow equation at the first stage and the GIC constraints at the second stage).
Our goal is to find an approximate solution to \eqref{DRO-Simple} with the help of the following relaxed DRO model \cite{ryu2020pscc}:

\begin{small}
	\begin{align}
	\min_{(\zb, \yb) \in \mathcal{S}} \ & \Big\{ \qb^{\text{\tiny T}} \zb + \bb^{\text{\tiny T}} \yb  + \sup_{\mathbb{P} \in \mathcal{D}} \mathbb{E}_{\mathbb{P}}[ \mathcal{Q} (\zb, \yb, \tilxi) ] \Big\},
	\label{DROLB-Simple}
	\end{align}
\end{small}	
where $\mathcal{S}$ is a feasible region constructed by relaxing the AC power flow constraints to the second-order conic ones and $\mathcal{Q} (\zb, \yb, \tilxi)$ is obtained by solving the following linear program (LP):

\vspace{-2mm}
\begin{small}
	\begin{align}
	\mathcal{Q} (\zb, \yb, \tilxi) = \min_{\xb \in \mathcal{X}(\zb, \yb, \tilxi)} \ \cb^{\text{\tiny T}} \xb,  \label{DROLB-Simple:stg2}
	\end{align}
\end{small}	
where $\mathcal{X}(\zb, \yb, \tilxi) = \{ \xb \in \mathbb{R}^{n} \ : \ \Ab \xb + \Bb \zb + \Cb \yb \geq \db (\tilxi) \}$.
Note that an optimal solution $(\bm{z}^*, \bm{y}^*) \in \mathcal{S}$ obtained by \eqref{DROLB-Simple} may be infeasible for the DRO model \eqref{DRO-Simple}, that is, $(\bm{z}^*, \bm{y}^*) \notin \mathcal{A}$.
Thus, we discuss how to find feasible transmission grid operations $(\bm{\hat{z}}, \bm{\hat{y}}) \in \mathcal{A}$ in Section \ref{subsec:ACFeasible}.
}

We start by reformulating the relaxed DRO model. 
{\color{black}
In the following lemma, we utilize the standard duality-based reformulation technique as in \cite{velloso2020distributionally}.
}

\begin{lemma} \label{lemma-1}
Problem \eqref{DROLB-Simple} is equivalent to 

\vspace{-2mm}
\begin{small}
    \begin{align}
\min_{(\zb,\yb) \in \mathcal{S}} \ & \qb^{\text{\tiny T}} \zb + \bb^{\text{\tiny T}} \yb + \mub^{\text{\tiny T}} \lambdab  + \bigg\{ \max_{\tilxi \in \Xi} \  \mathcal{Q} (\zb, \yb, \tilxi) -  \lambdab^{\text{\tiny T}} \tilxi  \bigg\}, \label{DROLB-Simple:Reform1}
    \end{align}    
\end{small}
where $\mub=[\mu_{\text{\tiny E}} \ \mu_{\text{\tiny N}}]^{\text{\tiny T}}$, $\lambdab=[\lambda_{\text{\tiny E}} \  \lambda_{\text{\tiny N}}]^{\text{\tiny T}}$, and $\tilxi=[\tilxiE \ \tilxiN]^{\text{\tiny T}}$.
\end{lemma}
\begin{proof}
{\color{black}
The worst-case expected value $\sup_{\mathbb{P} \in \mathcal{D}} \mathbb{E}_{\mathbb{P}}[ \mathcal{Q} (\zb, \yb, \tilxi) ]$ can be rewritten as follows:

\vspace{-2mm}
\begin{small}
\begin{align*}
\max \ & \int_{\tilxi \in \Xi}  \mathcal{Q} (\zb, \yb, \tilxi) \ d\mathbb{P} \\
\mbox{s.t.} \
& \int_{\tilxi \in \Xi} \tilxiE \ d\mathbb{P} = \mu_{\text{\tiny E}}, \ \int_{\tilxi \in \Xi} \tilxiN \ d\mathbb{P} = \mu_{\text{\tiny N}}, \ \int_{\tilxi \in \Xi} \ d\mathbb{P} = 1. 
\end{align*}
\end{small}
By taking the dual, we obtain

\vspace{-2mm}
\begin{small}
\begin{align*}
\min \ & \mu_{\text{\tiny E}} \lambda_{\text{\tiny E}} + \mu_{\text{\tiny N}} \lambda_{\text{\tiny N}} + \eta \\
\mbox{ s.t. } \ & \lambda_{\text{\tiny E}} \tilxiE + \lambda_{\text{\tiny N}}  \tilxiN  + \eta \geq  \ \mathcal{Q} (\zb, \yb, \tilxi), \ \forall \tilxi \in \Xi,
\end{align*}
\end{small}
where $\lambda_{\text{\tiny E}}$, $\lambda_{\text{\tiny N}}$, and $\eta$ are dual variables.
At optimality, we have $\eta^* = \max_{\tilxi \in \Xi} \ \{ \mathcal{Q} (\zb,\yb, \tilxi) - \lambda_{\text{\tiny E}} \tilxiE - \lambda_{\text{\tiny N}}  \tilxiN \}$, which leads to the proposed formulation.
}
\end{proof}

\begin{lemma} \label{lemma-2}
An optimal solution $\tilxi^*$ to the inner max problem in \eqref{DROLB-Simple:Reform1} is an extreme point of the polyhedral support set $\Xi$, namely, $\tilxi^* \in \{\hat{\bm{\xi}}^{\ell}, \ \forall \ell \in [N] \}$.
\end{lemma}
\begin{proof}
By utilizing the dual formulation of the second-stage problem \eqref{DROLB-Simple:stg2}, we can rewrite the inner max problem in \eqref{DROLB-Simple:Reform1} as

\vspace{-2mm}
\begin{small}
\begin{subequations}
\label{DROLB-Simple:innermax}
\begin{align}
\max_{ \tilxi \in \Xi  } \ & (\db (\tilxi) - \Bb \zb - \Cb \yb)^{\text{\tiny T}} \betab -  \lambdab^{\text{\tiny T}} \tilxi \\
\mbox{s.t. } \ & \Ab^{\text{\tiny T}} \betab = \cb, \ \ \betab \in \mathbb{R}^m_+,
\end{align}    
\end{subequations}
\end{small}
where $\betab$ is a dual vector corresponding to constraints in primal LP \eqref{DROLB-Simple:stg2}.
Note that formulation \eqref{DROLB-Simple:innermax} includes (i) bilinear terms, namely, inner product of $\betab$ and $\tilxi$, and (ii) a convex objective function in $\tilxi$ since it is the maximum of infinitely many linear functions.
Because of its convexity, for a given feasible solution ($\zb, \yb, \lambdab$), an optimal solution to the inner max problem \eqref{DROLB-Simple:innermax} is an extreme point of the support set $\Xi$.
\end{proof}

The relaxed DRO model \eqref{DROLB-Simple} is solvable by utilizing the CCG algorithm. 
Specifically, by Lemma \ref{lemma-2}, we can rewrite the relaxed DRO model \eqref{DROLB-Simple} as follows:

\vspace{-2mm}
\begin{small}
\begin{subequations}
\label{DRO_Poly_Robust}
\begin{align}
\min_{(\zb, \yb) \in \mathcal{S}} \ & \qb^{\text{\tiny T}} \zb + \bb^{\text{\tiny T}} \yb + \mub^{\text{\tiny T}} \lambdab + \eta \\
\mbox{s.t.} \ 
& \eta \geq \cb^{\text{\tiny T}} \xb^{\ell} - \lambdab^{\text{\tiny T}} \hat{\bm{\xi}}^{\ell}, \ \forall \ell \in [N], \label{DRO_Poly_Robust-1} \\
& \Ab \xb^{\ell} + \Bb \zb + \Cb \yb \geq \db(\hat{\bm{\xi}}^{\ell}), \ \forall \ell \in [N], \label{DRO_Poly_Robust-2}
\end{align}
\end{subequations}
\end{small}
where $\xb^{\ell}$ is a recourse decision vector defined for every scenario $\ell \in [N]$, each of which corresponds to an extreme point $\hat{\bm{\xi}}^{\ell}$ of the support set $\Xi$.
The main idea of the CCG algorithm is that one might not need all scenarios $[N]$ to find an optimal solution to problem \eqref{DRO_Poly_Robust}. 
Thus, it first relaxes all constraints and dynamically generates a set of constraints \eqref{DRO_Poly_Robust-1} and \eqref{DRO_Poly_Robust-2} with recourse decision vector $\xb^{\ell}$ for an identified scenario $\ell \in [N]$.
The complexity of this algorithm is $O(N)$, where $N$ is the total number of extreme points of the support set $\Xi$. If the support set $\Xi$ contains infinitely many extreme points (e.g., a half-circle as depicted in Figure \ref{fig:SupportSets}), this algorithm does not guarantee finite convergence. 
Readers interested in a solution approach for solving the relaxed DRO model \eqref{DROLB-Simple} under the nonlinear, half-circle support set should see \cite{ryu2020pscc}.

In the remainder of this section, we describe a MISOCP reformulation of the relaxed DRO model \eqref{DROLB-Simple} under the triangle support set in Section \ref{subsec:triangle}.
Then, in Section \ref{subsec:AcceleratedCCG} we propose an accelerated CCG algorithm that improves the computational performance of CCG, which is the main result of this paper.
In Section \ref{subsec:ACFeasible} we discuss how to find feasible transmission grid operations $(\bm{\hat{z}}, \bm{\hat{y}}) \in \mathcal{A}$.

\subsection{Triangle support set: MISOCP reformulation}
\label{subsec:triangle}
In this section we focus on a triangle support set $\Xi^3 = \text{conv} \{ \hat{\bm{\xi}}^{\ell}, \ \forall \ell \in [3] \}$ (see Figure \ref{fig:SupportSets} for an example). 
Indeed the relaxed DRO model \eqref{DROLB-Simple} under $\Xi^3$ can be solved by CCG.
Alternatively, we derive an exact MISOCP reformulation (as discussed in Proposition \ref{prop:monolithic}), which can be solved efficiently by using off-the-shelf commercial optimization solvers such as CPLEX or Gurobi. 
\begin{lemma}
Under the triangle support set $\Xi^3$, the worst-case expected value $\sup_{\mathbb{P} \in \mathcal{D}} \mathbb{E}_{\mathbb{P}}[ \mathcal{Q} (\zb, \yb, \tilxi) ]$ in the formulation \eqref{DROLB-Simple} is equivalent to
	
	\vspace{-2mm}
	\begin{small}
	\begin{subequations}
		\label{PrimalLP}
		\begin{align}
		\max_{p \in \mathbb{R}^3_+} \ & \sum_{\ell=1}^3 \mathcal{Q}(\zb, \yb, \hat{\bm{\xi}}^{\ell}) p_{\ell} \\
		\mbox{s.t.} \ 
		& \sum_{\ell=1}^3 \hat{\xi}_{\text{\tiny E}}^{\ell} p_{\ell} = \mu_{\text{\tiny E}}, \ 
		\sum_{\ell=1}^3 \hat{\xi}_{\text{\tiny N}}^{\ell} p_{\ell} = \mu_{\text{\tiny N}}, \ 
		\sum_{\ell=1}^3 p_{\ell} = 1,  \label{PrimalLP-3}
		\end{align}
	\end{subequations}
	\end{small}
	where $\{ \hat{\bm{\xi}}^1, \hat{\bm{\xi}}^2, \hat{\bm{\xi}}^3 \}$ are the extreme points of the set $\Xi^3$.
\end{lemma} 
\begin{proof}
{\color{black}
For fixed $(\zb, \yb) \in \mathcal{S}$, the function $\mathcal{Q}(\zb, \yb, \tilxi)$ is convex in $\tilxi$ by Lemma \ref{lemma-2}. Let $\mathcal{H}:= \text{conv} \{(\tilxi, t) : \tilxi \in \Xi^3, \ t =\mathcal{Q}(\zb, \yb, \tilxi) \}$.
Since taking expectation can be viewed as a convex combination, it follows that $(\mub, \mathcal{Q}(\zb, \yb, \mub)) \in \mathcal{H}$.
Since the set $\mathcal{D}$ is a mean-support ambiguity set where the support set $\Xi^3$ is a simplex with 3 extreme points $\hat{\bm{\xi}}^1, \hat{\bm{\xi}}^2, \hat{\bm{\xi}}^3$, there is an unique convex combination of the extreme points which yields $\mub$.
Therefore, we have $\sup_{\Pb \in \mathcal{D}} \Eb_{\Pb}[\mathcal{Q}(\zb, \tilxi)] = \sup \{ t | (\mub, t) \in \mathcal{H} \} $, where $\mub$ is a convex combination of the 3 extreme points, $\mub = \hat{\bm{\xi}}^1 p_1 + \hat{\bm{\xi}}^2 p_2 + \hat{\bm{\xi}}^3 p_3$, and $t = \mathcal{Q}(\zb, \hat{\bm{\xi}}^1) p_1 + \mathcal{Q}(\zb, \hat{\bm{\xi}}^2) p_2 + \mathcal{Q}(\zb, \hat{\bm{\xi}}^3) p_3$.
}
\end{proof}

\begin{corollary}
	The worst-case distribution $(p^*_1, p^*_2, p^*_3)$ is uniquely determined by the unique convex combination of extreme points, which yields $\mub$.
	
	\vspace{-2mm}
	\begin{small}
	\begin{align*}
	\begin{bmatrix}
	p^*_1 \\ p^*_2 \\ p^*_3
	\end{bmatrix}
	=
	\begin{bmatrix}
	\hat{\xi}_{\text{\tiny E}}^1 & \hat{\xi}_{\text{\tiny E}}^2 & \hat{\xi}_{\text{\tiny E}}^3 \\
	\hat{\xi}_{\text{\tiny N}}^1 & \hat{\xi}_{\text{\tiny N}}^2 & \hat{\xi}_{\text{\tiny N}}^3 \\
	1 & 1 & 1 \\		
	\end{bmatrix}^{-1}
	\begin{bmatrix}
	\mu_{\text{\tiny E}} \\ \mu_{\text{\tiny N}} \\ 1
	\end{bmatrix}
	\end{align*}
	\end{small}
	As long as $\mub \in \Xi^3$, we have $p^*_1, p^*_2, p^*_3 \geq 0$.
\end{corollary}

\begin{proposition} \label{prop:monolithic}
	The relaxed DRO model \eqref{DROLB-Simple} under the triangle support set ($\Xi^3$) is equivalent to the following MISOCP:
	
	\vspace{-2mm}
	\begin{small}
	\begin{subequations}
	\label{SLP_Primal}		
	\begin{align}
	\min_{(\zb,\yb) \in \mathcal{S}} \ & \qb^{\text{\tiny T}} \zb + \bb^{\text{\tiny T}} \yb + \sum_{\ell=1}^3 p^*_{\ell} ( \cb^{\text{\tiny T}} \xb^{\ell} ) \\
	\mbox{s.t.} \
	& \Ab \xb^{\ell} + \Bb \zb + \Cb \yb \geq \db(\hat{\bm{\xi}}^{\ell}), \  \forall \ell \in [3].
	\end{align}	
	\end{subequations}
	\end{small}
\end{proposition}

\begin{remark} \label{remark-1}
	If $\mub$ coincides with an extreme point $\hat{\bm{\xi}}^{\ell^*}$ of $\Xi^3$, problem \eqref{SLP_Primal} is equivalent to a deterministic optimization problem since $p_{\ell^*} = 1$.
	If $\mub$ lies in an edge of $\Xi^3$, problem \eqref{SLP_Primal} is equivalent to a stochastic program with $2$ scenarios. In this paper we consider only the case when $\mub$ lies in the interior of $\Xi^3$.
\end{remark}

\subsection{Accelerated CCG algorithm} \label{subsec:AcceleratedCCG}
In this section we propose a \textit{new} accelerated CCG algorithm that solves the relaxed DRO model \eqref{DROLB-Simple} under the polyhedral support set $\Xi$.
The main idea of the proposed algorithm is to incorporate the MISOCP reformulation in Proposition \ref{prop:monolithic} into the CCG algorithm.
Specifically, given any polyhedral support set $\Xi$, one can always construct a triangle support set $\Xi^3$$\subseteq$$\Xi$ that contains $\mub$.
Therefore, the MISOCP reformulation \eqref{SLP_Primal} under $\Xi^3$$\subseteq$$\Xi$ provides a lower bound on the optimal value of problem \eqref{DROLB-Simple} under $\Xi$.
By incorporating the obtained lower bound into the CCG algorithm, we wish to improve the computational performance.

Suppose that ($\hat{\zb}, \hat{\yb}, \hat{ \xb}^{\ell}, \forall \ell \in [3]$) is an optimal solution and $\text{OBJ}=\qb^{\text{\tiny T}} \hat{\zb} + \bb^{\text{\tiny T}} \hat{\yb} + \sum_{\ell=1}^3 p^*_{\ell} ( \cb^{\text{\tiny T}} \hat{\xb}^{\ell})$ is an optimal value of formulation \eqref{SLP_Primal}.
Since $\text{OBJ}$ is a lower bound on the optimal value of problem \eqref{DROLB-Simple}, we can set OBJ as an initial lower bound in the CCG algorithm.
Then, we plug ($\hat{\zb}, \hat{\yb}, \hat{\xb}^{\ell}, \forall \ell \in [3]$) into the master problem, which leads to the following problem:

\vspace{-2mm}
\begin{small}
\begin{subequations}
\label{Modified_Master_1}
\begin{align}
\min \ & \qb^{\text{\tiny T}} \hat{\zb} + \bb^{\text{\tiny T}} \hat{\yb} + \mu_{\text{\tiny E}} \lambda_{\text{\tiny E}} + \mu_{\text{\tiny N}} \lambda_{\text{\tiny N}} + \eta \\ 
\mbox{ s.t. } \ 
& \eta \geq \cb^{\text{\tiny T}} \hat{\xb}^{\ell} - \lambda_{\text{\tiny E}} \hat{\xi}_{\text{\tiny E}}^{\ell} - \lambda_{\text{\tiny N}} \hat{\xi}_{\text{\tiny N}}^{\ell} , \ \forall \ell \in \mathcal{T}, \\
& \Ab \hat{\xb}^{\ell} + \Bb \hat{\zb} + \Cb \hat{\yb} \geq \db(\hat{\bm{\xi}}^{\ell}), \ \forall \ell \in \mathcal{T},
\end{align}
\end{subequations}
\end{small}
where $\mathcal{T}$ represents a set of scenarios and
$(\lambda_{\text{\tiny E}}, \lambda_{\text{\tiny N}}, \eta)$ are decision variables.
In the following proposition, we describe a closed-form solution of $(\lambda_{\text{\tiny E}}, \lambda_{\text{\tiny N}}, \eta)$.

\begin{proposition} \label{prop:closedform}
For $\mathcal{T}=[3]$, problem \eqref{Modified_Master_1} admits the following closed-form solution:

\vspace{-2mm}
\begin{small}
\begin{align*}
\begin{bmatrix}
\lambda^*_{\text{\tiny E}} \\ \lambda^*_{\text{\tiny N}} \\ \eta^*
\end{bmatrix}
=
\begin{bmatrix}
\hat{\xi}_{\text{\tiny E}}^1 & \hat{\xi}_{\text{\tiny N}}^1 & 1 \\
\hat{\xi}_{\text{\tiny E}}^2 & \hat{\xi}_{\text{\tiny N}}^2 & 1 \\
\hat{\xi}_{\text{\tiny E}}^3 & \hat{\xi}_{\text{\tiny N}}^3 & 1 \\
\end{bmatrix}^{-1}
\begin{bmatrix}
\cb^{\text{\tiny T}} \hat{\xb}^1 \\ \cb^{\text{\tiny T}} \hat{\xb}^2 \\ \cb^{\text{\tiny T}} \hat{\xb}^3
\end{bmatrix}.
\end{align*}
\end{small}
\end{proposition}
\begin{proof}
Given an optimal solution ($\hat{\zb}, \hat{\yb}, \hat{\xb}^{\ell}, \forall \ell \in [3]$), we can rewrite problem \eqref{SLP_Primal} as follows:

\vspace{-2mm}
\begin{small}
\begin{subequations}
\begin{align}
\qb^{\text{\tiny T}} \hat{\zb} + \bb^{\text{\tiny T}} \hat{\yb} + \max_{p \in \mathbb{R}^3_+} \ & \sum_{\ell=1}^3 p_{\ell} (\cb^{\text{\tiny T}} \hat{\xb}^{\ell})  \nonumber \\    
\mbox{s.t.} \ 
& \hat{\xi}_{\text{\tiny E}}^1 p_1 + \hat{\xi}_{\text{\tiny E}}^2 p_2 + \hat{\xi}_{\text{\tiny E}}^3 p_3 = \mu_{\text{\tiny E}}, \label{const-mue} \\
& \hat{\xi}_{\text{\tiny N}}^1 p_1 + \hat{\xi}_{\text{\tiny N}}^2 p_2 + \hat{\xi}_{\text{\tiny N}}^3 p_3 = \mu_{\text{\tiny N}}, \label{const-mun}\\
& p_1 + p_2 + p_3 = 1. \label{const-prob}
\end{align}
\end{subequations}
\end{small}
By taking the dual to the inner max problem, we obtain

\vspace{-2mm}
\begin{small}
\begin{align*}
\qb^{\text{\tiny T}} \hat{\zb} + \bb^{\text{\tiny T}} \hat{\yb} + \min \ & \mu_{\text{\tiny E}} \lambda_{\text{\tiny E}} + \mu_{\text{\tiny N}} \lambda_{\text{\tiny N}} + \eta   \\     
\mbox{s.t.} \ 
& \hat{\xi}_{\text{\tiny E}}^{\ell} \lambda_{\text{\tiny E}} + \hat{\xi}_{\text{\tiny N}}^{\ell} \lambda_{\text{\tiny N}} + \eta \geq \cb^{\text{\tiny T}} \hat{\xb}^{\ell}, \ \forall \ell \in [3],
\end{align*}
\end{small}
where $\lambda_{\text{\tiny E}}$, $\lambda_{\text{\tiny N}}$, and $\eta$ are dual variables for constraints \eqref{const-mue}, \eqref{const-mun}, and \eqref{const-prob}, respectively.
Suppose that $p^*_{\ell}, \ \forall \ell \in [3]$, and $(\lambda^*_{\text{\tiny E}}, \lambda^*_{\text{\tiny N}}, \eta^*)$ are primal and dual optimal solutions. Then the following complementary slackness condition should hold:

\vspace{-2mm}
\begin{align*}
(\hat{\xi}_{\text{\tiny E}}^k \lambda^*_{\text{\tiny E}} + \hat{\xi}_{\text{\tiny N}}^k \lambda^*_{\text{\tiny N}} + \eta^* - \cb^{\text{\tiny T}} \hat{\xb}^k) p^*_k = 0, \ \forall k \in [3].  
\end{align*}
By Remark \ref{remark-1}, we have $p^*_k > 0, \ \forall k \in [3]$, which leads to

\vspace{-2mm}

\begin{align*}
\hat{\xi}_{\text{\tiny E}}^k \lambda^*_{\text{\tiny E}} + \hat{\xi}_{\text{\tiny N}}^k \lambda^*_{\text{\tiny N}} + \eta^* = \cb^{\text{\tiny T}} \hat{\xb}^k, \ \forall k \in [3].  \hspace{3.22cm} \qedhere
\end{align*}
\end{proof}

We conclude this section by formally outlining the accelerated CCG algorithm in Algorithm \ref{algo:ModifiedCCG}.

\begin{algorithm}[!h]
    \small
	\caption{\small Accelerated CCG algorithm for solving problem \eqref{DROLB-Simple}}
	\label{algo:ModifiedCCG}
	\begin{algorithmic}[1]
	    \STATE Solve problem \eqref{SLP_Primal} with $\Xi^3 \subseteq \Xi$, and record an optimal solution ($\hat{\zb}$, $\hat{\yb}$, $\hat{\xb}^{\ell}, \forall \ell \in [3]$).
		\STATE Set LB=$\qb^{\text{\tiny T}} \hat{\zb} + \bb^{\text{\tiny T}} \hat{\yb}+ \sum_{\ell=1}^3 p^*_{\ell} ( \cb^{\text{\tiny T}} \hat{\xb}^{\ell})$, UB=$\infty$, $t=3$ and $\mathcal{T}=[3]$.
		\STATE Set ($\lambda_{\text{\tiny E}}^*, \lambda_{\text{\tiny N}}^*, \eta^*$) = closed-form solution in Proposition \ref{prop:closedform}.
        \STATE Solve the following problem:
		\begin{align*}
		\mathcal{Z}(\hat{\zb}, \hat{\yb},\lambda_{\text{\tiny E}}^*, \lambda_{\text{\tiny N}}^*) = \max_{\tilxi \in \Xi } \ \bigg\{ \mathcal{Q}(\hat{\zb}, \hat{\yb}, \tilxi) - \lambda_{\text{\tiny E}}^*  \tilxiE - \lambda_{\text{\tiny N}}^*  \tilxiN \bigg\}.
		\end{align*}		
		 -- Record an optimal $\tilxi^*$ and the optimal value $\mathcal{Z}(\hat{\zb}, \hat{\yb},\lambda_{\text{\tiny E}}^*, \lambda_{\text{\tiny N}}^*)$.
		\STATE Update UB by 
		\begin{align*}
		\text{UB} = \min \{ \text{UB}, \ \qb^{\text{\tiny T}} \hat{\zb} + \bb^{\text{\tiny T}} \hat{\yb} + \mu_{\text{\tiny E}} \lambda_{\text{\tiny E}}^* + \mu_{\text{\tiny N}} \lambda_{\text{\tiny N}}^* + \mathcal{Z}(\hat{\zb}, \hat{\yb},\lambda_{\text{\tiny E}}^*, \lambda_{\text{\tiny N}}^*) \}.
		\end{align*}
		\IF {$(\text{UB}-\text{LB})/\text{UB} \leq \epsilon$} 
		\STATE Stop and return $\hat{\zb},\hat{\yb}$ as an optimal solution.
		\ELSE 
		\STATE Update $\hat{\xi}_{\text{\tiny E}}^{t+1}=\tilde{\xi}_{\text{\tiny E}}^*$, $\hat{\xi}_{\text{\tiny N}}^{t+1}=\tilde{\xi}_{\text{\tiny N}}^*$, $\mathcal{T} = \mathcal{T} \cup \{t+1\}$ and $t=t+1$.
		\ENDIF			
		\STATE Update LB by solving the following master problem:
		\begin{subequations}\label{Algo_LB}
			\begin{align}
			\text{LB} = \min_{(\zb,\yb) \in \mathcal{S}} \ & \qb^{\text{\tiny T}} \zb + \bb^{\text{\tiny T}} \yb + \mu_{\text{\tiny E}} \lambda_{\text{\tiny E}} + \mu_{\text{\tiny N}} \lambda_{\text{\tiny N}} + \eta \\ 
			\mbox{ s.t. } \ 
			& \eta \geq \cb^{\text{\tiny T}} \xb^{\ell} - \lambda_{\text{\tiny E}} \hat{\xi}_{\text{\tiny E}}^{\ell} - \lambda_{\text{\tiny N}} \hat{\xi}_{\text{\tiny N}}^{\ell} , \ \forall \ell \in \mathcal{T}, \\
			& \Ab \xb^{\ell} + \Bb \zb + \Cb \yb \geq \db(\hat{\bm{\xi}}^{\ell}), \ \forall \ell \in \mathcal{T}.
			\end{align}
		\end{subequations}
		-- Record an optimal solution $\hat{\zb}, \hat{\yb}, \lambda_{\text{\tiny E}}^*, \lambda_{\text{\tiny N}}^*$, and $\eta^*$.
    \STATE Go to step 4.		
	\end{algorithmic}
\end{algorithm}

\subsection{Recovery of AC and GIC feasible solution}
\label{subsec:ACFeasible}
In this section we discuss how to obtain a feasible solution to the original DRO model \eqref{DRO-Simple} with nonlinear and nonconvex AC and GIC constraints. To this end, we assume that the binary decision vector $\hat{\zb}$ is given.
{\color{black}
For example, $\hat{\zb}$ can be chosen as an optimal solution to the relaxed DRO model \eqref{DROLB-Simple}. If the chosen $\hat{\zb}$ leads to the AC infeasibility when $\tilxi=\mub=0$ in \eqref{DRO-Simple:Reform1}, one will need to solve a new DRO model by a modified Algorithm \ref{algo:ModifiedCCG} where $\mathcal{S}$ in \eqref{Algo_LB} is replaced with $\mathcal{A}$ from \eqref{DRO-Simple}. The resulting optimization problem in  \eqref{Algo_LB} will be a nonconvex mixed-integer nonlinear program, which can be solved either using local branch-and-bound methods (such as in Juniper.jl \cite{kroger2018juniper}) or to global optimality at the cost of larger run times.
}

Given $\hat{\zb}$, the DRO model \eqref{DRO-Simple} can be rewritten as

\vspace{-2mm}
\begin{small}
\begin{align}
\qb^{\text{\tiny T}} \hat{\zb} + \min_{ \yb \in \mathcal{A}(\hat{\zb})} \ & \bb^{\text{\tiny T}} \yb + \mub^{\text{\tiny T}} \lambdab  + \bigg\{ \max_{\tilxi \in \Xi} \  \mathcal{V} (\hat{\zb}, \yb, \tilxi) -  \lambdab^{\text{\tiny T}} \tilxi  \bigg\},    
\label{DRO-Simple:Reform1}
\end{align}
\end{small}
where $\mathcal{A}(\hat{\zb})$ is a feasible region of the first-stage problem \eqref{DRO:stg1} given $\hat{\zb}$. Since $\mathcal{V} (\hat{\zb}, \yb, \tilxi)$ is the optimal value of the nonlinear program \eqref{DRO:stg2}, Lemma \ref{lemma-2} does not hold.
Therefore, it becomes even more challenging to solve the inner max problem in \eqref{DRO-Simple:Reform1}, which is required by the CCG algorithm.

Instead of solving the inner max problem in \eqref{DRO-Simple:Reform1} to obtain UB in Algorithm \ref{algo:ModifiedCCG}, we solve the following restricted inner max problem:

\vspace{-2mm}
\begin{small}
\begin{align}
\max_{\tilxi \in \Xi^{\text{\tiny S}}} \  \mathcal{V} (\hat{\zb}, \yb, \tilxi) -  \lambdab^{\text{\tiny T}} \tilxi ,
\label{restricted_inner_max}
\end{align}
\end{small}
where $\Xi^{\text{\tiny S}}$ is obtained by sampling the uncertain parameter $\tilxi$ from the support set $\Xi$. Note that as we increase the number of sampling points, the obtained solution converges to an optimal solution of the formulation \eqref{DRO-Simple:Reform1}.
With this approach, we can utilize the accelerated CCG algorithm to obtain AC feasible transmission grid operations $(\hat{\zb}, \yb(\hat{\zb}) ) \in \mathcal{A}$.

\section{Numerical Experiments} \label{sec:numerical}
We conduct numerical experiments to show (i) the computational performances of our solution approaches and (ii) the performances of distributionally robust transmission grid operations produced by our DRO model.
To this end, we first generate several instances as described in Section \ref{Instance_Description}. Next, in Section \ref{Computational_Performances}, we compare our new solution approaches with the classical CCG algorithm in \cite{ryu2020pscc}.
Then, in Section \ref{Performance_comparison}, we compare solutions obtained by our DRO approach with the other deterministic approaches.\\
\indent
All computations were performed on a personal laptop with 2.3 GHz Intel Core i7 CPU and 16 GB of memory.  
The relaxed DRO model \eqref{DROLB-Simple} and its corresponding algorithms were implemented in \texttt{C++} and solved by using \texttt{Gurobi v9.0.1}. 
Also, problem \eqref{DRO-Simple:Reform1} was solved by using \texttt{Ipopt v0.6.1}, which provides AC and GIC feasible solutions. 

\subsection{Instance description} \label{Instance_Description}
\subsubsection{Transmission grid systems}
We consider the widely-examined ``epri-21'' and ``uiuc-150'' systems, designed specifically for GMD studies. 
Figure \ref{fig:Transmission_Grids} shows simplified diagrams of the epri-21 system located near Atlanta, GA (top), and the uiuc-150 system located near Nashville, TN (bottom).
The epri-21 system has $19$ buses, $7$ generators, $15$ transmission lines, $16$ transformers, and $8$ substations.
The uiuc-150 system has $150$ buses, $27$ generators, $157$ transmission lines, $61$ transformers, and $98$ substations. 
In these diagrams, the blue and green lines are 500 kV and 345 kV transmission lines, respectively.
{\color{black} The DC model's input parameters of the epri-21 and the uiuc-150 systems are given in \cite{horton2012test} and \cite{birchfield2016statistical}, respectively, while the AC model's input parameters are available from \url{https://github.com/lanl-ansi/PowerModelsGMD.jl}.}
{\color{black} The remaining parameters for our DRO model are described in Table \ref{tab:parameters}.}

\begin{table}[h!]
    \centering
    \caption{Additional parameters.}
    \begin{tabular}{cccc}
    \hline
    $\kappa^{\text{\tiny l}}$ & $\kappa^{\text{\tiny s}}$ & $\overline{v}^{\text{\tiny d}}$ & $\overline{I}^{\text{\tiny eff}}_e : e_{ij} \in \mathcal{E}^{\tau}$ \\ \hline
    $\$ 50,000/pu$ & $\$ 100,000/pu$ & $ 10,000 [V]$ & $ 2\frac{\overline{s}_e}{\min(\underline{v}_i, \underline{v}_j)} \text{[Amp]}$  \\ \hline
    \end{tabular}
    \label{tab:parameters}
\end{table}


\begin{figure}[h]
    \centering
    \includegraphics[width=\linewidth]{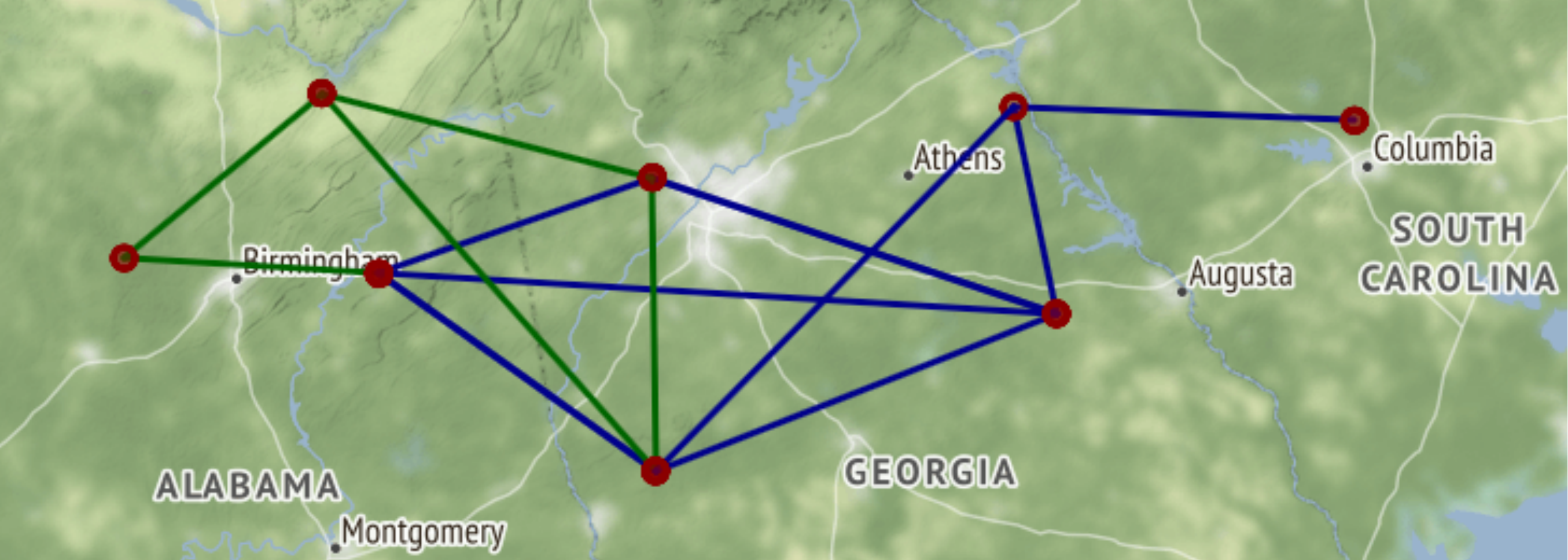}\\ \vspace{0.3cm}
    \includegraphics[width=\linewidth]{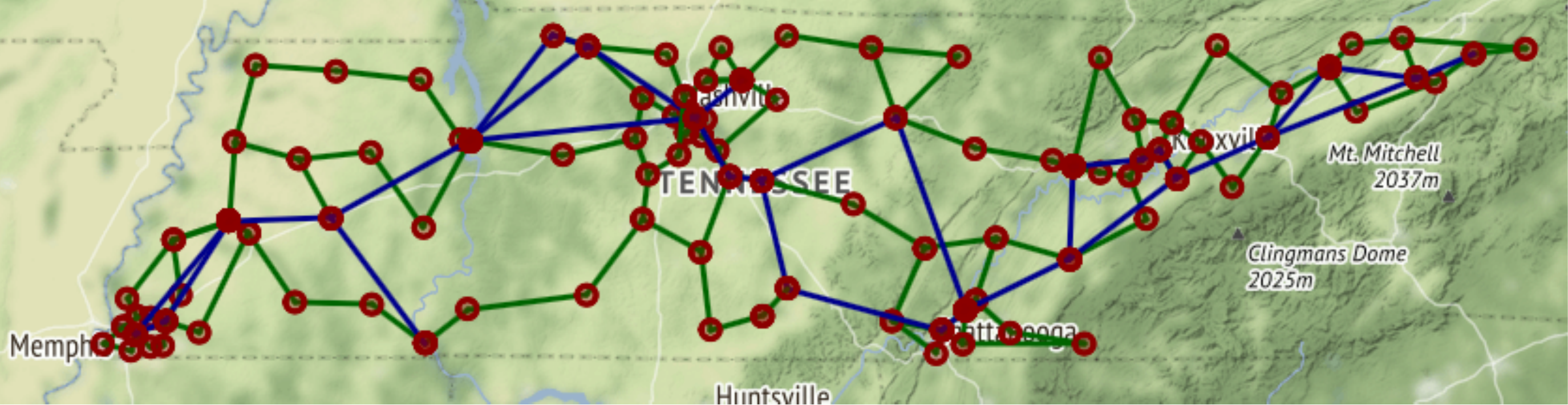}
    \caption{Benchmark cases for GMD studies: epri-21 (top) and uiuc-150 (bottom).}
    \label{fig:Transmission_Grids}
\end{figure}
\subsubsection{E-field data} \label{sec:E-field}
Recall that the inputs of the ambiguity set $\mathcal{D}$ in our DRO model are

\vspace{-2mm}
\begin{small}
\begin{align*}
& \mu_{\text{\tiny E}} = M \cos \delta^{\mu}, \ \
\mu_{\text{\tiny N}} = M \sin \delta^{\mu}, \ \
\Xi=\text{conv} \{ (\hat{\xi}^{\ell}_{\text{\tiny E}}, \hat{\xi}^{\ell}_{\text{\tiny N}}), \ \forall \ell \in [N] \},
\end{align*}
\end{small}
where $\hat{\xi}^{\ell}_{\text{\tiny E}} = R \cos \delta^{\ell}$ and $\hat{\xi}^{\ell}_{\text{\tiny N}} = R \sin \delta^{\ell}$.
Throughout this section we focus mainly on the \textit{pentagon} support set, whose $N=5$ extreme points are determined by $R$ and $(\delta^{\ell}, \forall \ell \in [5])=(0^{\circ},45^{\circ},90^{\circ},135^{\circ},180^{\circ})$ (see Figure \ref{fig:SupportSets} for an example). The reason for choosing the pentagon support set is explained in Section \ref{subsec:Ambiguity}. 
Moreover, based on historical data of the E-field (e.g., the maximum magnitude $R$ of the E-field in the Hydro-Qu\'ebec GMD event in 1989 was $8$ V/Km), we set various $R \in \{2.5, 5, 7.5, 10, 12.5\}$ [V/km].
For the mean values ($M$, $\delta^{\mu}$) of random magnitude $\tilde{M}$ and direction $\tilde{\delta}$ of E-field, respectively, we set $M \in \{0.5R, 0.7R\}$ [V/Km] and $\delta^{\mu} \in \{45^{\circ}, 67.5^{\circ}, 90^{\circ}, 112.5^{\circ}, 135^{\circ} \}$. 

In summary, for each transmission grid system (epri-21 or uiuc-150), we generate $50$ different instances based on the various $(R,M,\delta^{\mu})$ to test our solution approaches for the proposed DRO model.

\subsection{Computational Performances} \label{Computational_Performances}
In this section, we compare the computational performances of our solution approaches with those of the classical CCG algorithm. Specifically, using 50 instances described in Section \ref{Instance_Description}, we solve the relaxed DRO model \eqref{DROLB-Simple} (i) under the \textit{triangle} support set by the proposed MISOCP reformulation and classical CCG, and compare their computation times, and (ii) under the \textit{pentagon} support set by the proposed accelerated CCG (Algorithm \ref{algo:ModifiedCCG}) and classical CCG, and compare their computation times.

Figure \ref{fig:computation_triangle} shows computation times of the MISOCP reformulation and CCG that solve the relaxed DRO model \eqref{DROLB-Simple} under the triangle support set.
Each point under the 45 degree blue line represents an instance where the MISOCP reformulation computationally outperforms CCG.
Compared with CCG, the MISOCP reformulation is on average \textit{2.6} times faster for the epri-21 system and \textit{2.3} times faster for the uiuc-150 system.

Figure \ref{fig:computation_pentagon} shows computation times of the accelerated CCG and CCG that solve the relaxed DRO model \eqref{DROLB-Simple} under the pentagon support set.
On average, the accelerated CCG is \textit{2.6} times faster for the epri-21 system and \textit{2.4} times faster for the uiuc-150 system.

\begin{figure}[!h]
    \centering
    \includegraphics[scale=0.22]{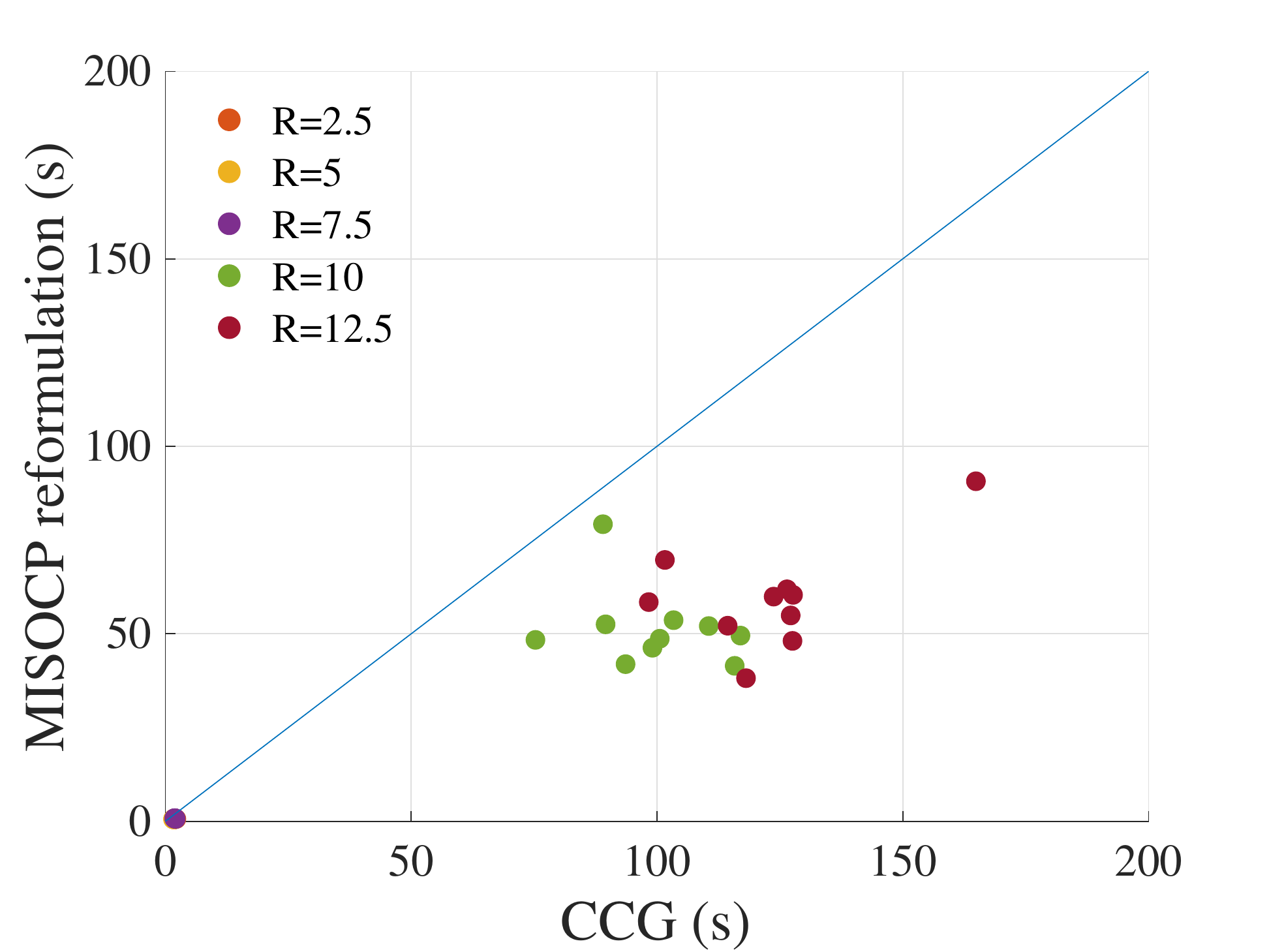}
    \includegraphics[scale=0.22]{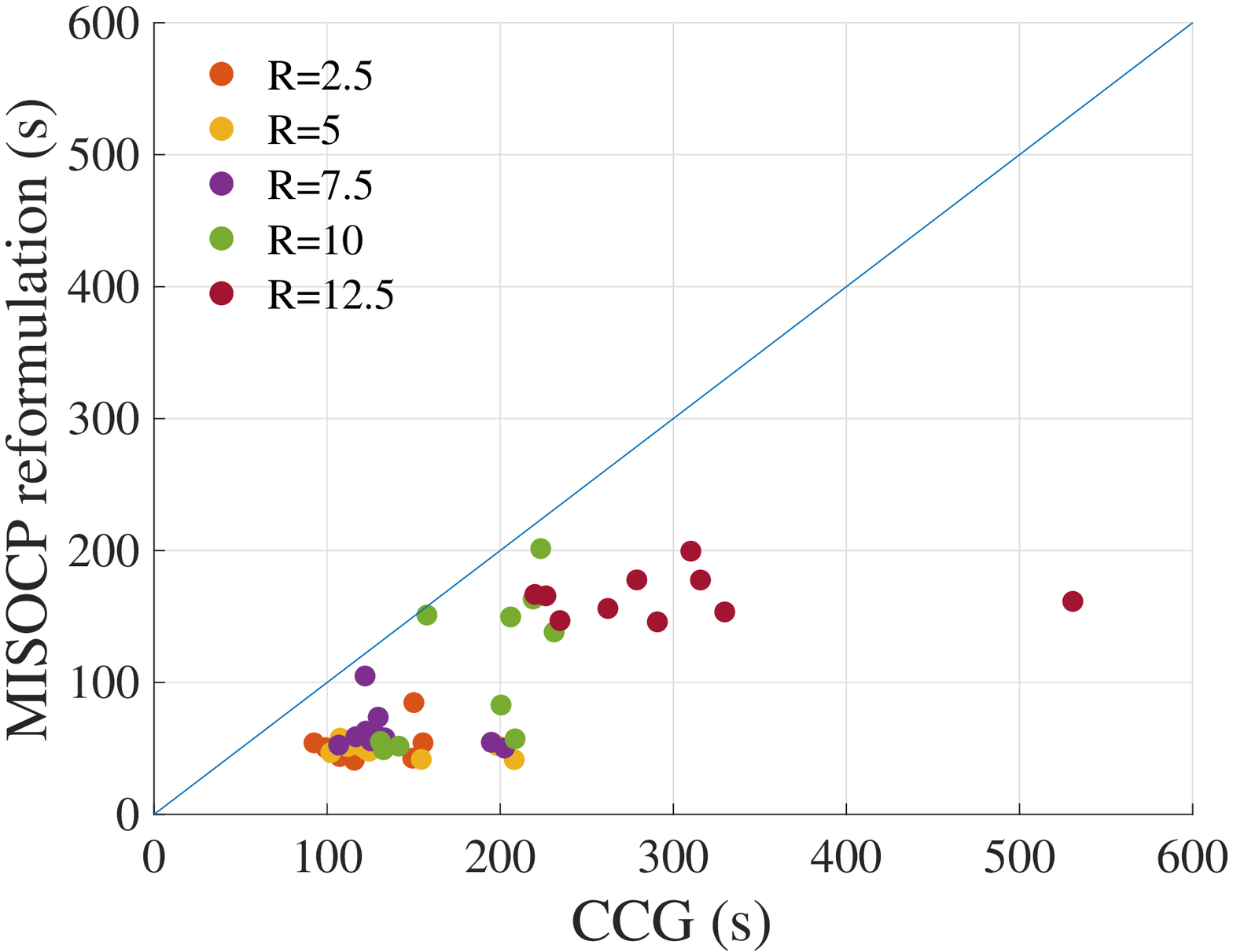}
    \caption{Computation (wall-clock) times of the MISOCP reformulation and CCG: epri-21 (left) and uiuc-150 (right).}
    \label{fig:computation_triangle}
\end{figure}

\begin{figure}[!h]
    \centering
    \includegraphics[scale=0.22]{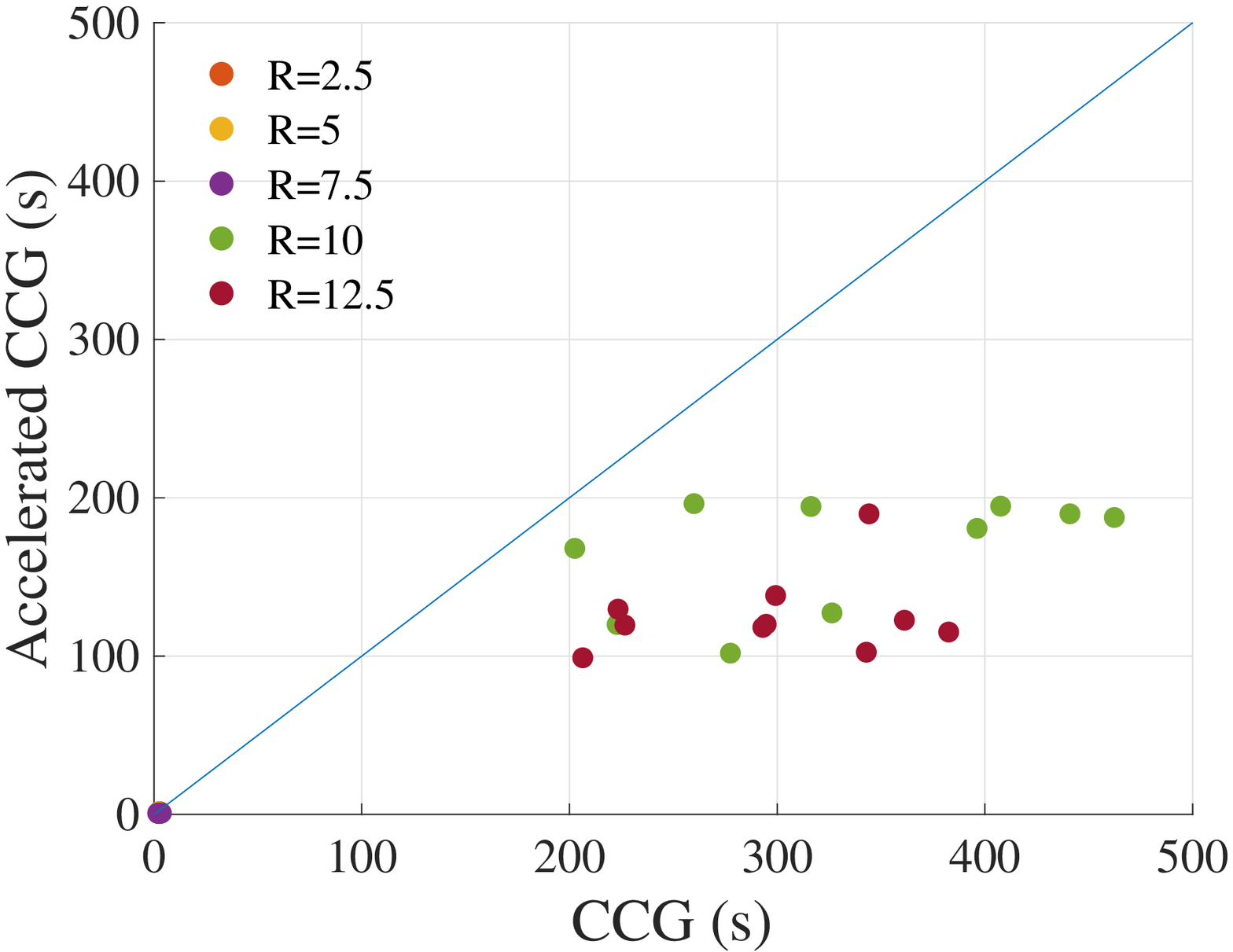}
    \includegraphics[scale=0.22]{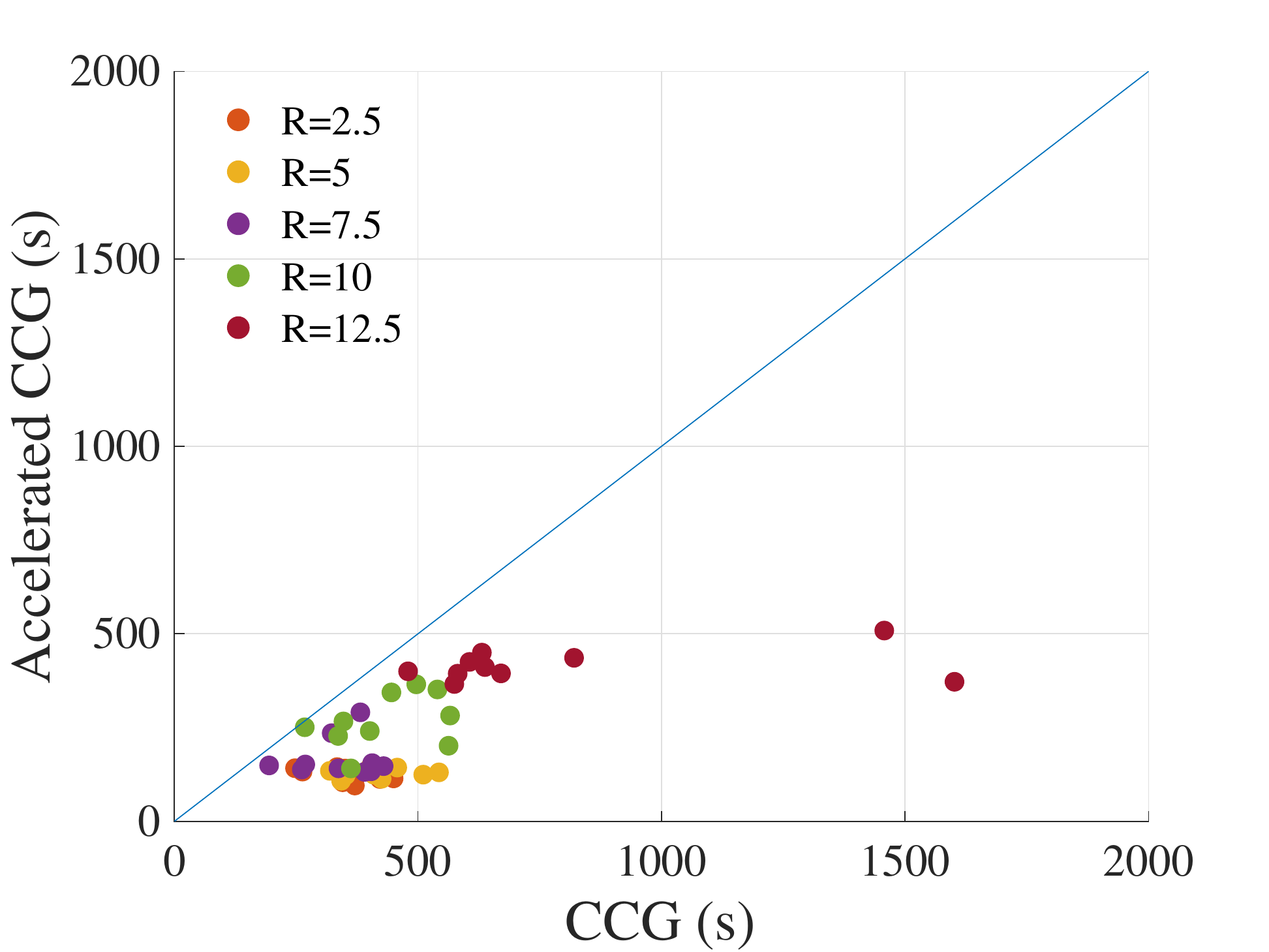}
    \caption{Computation (wall-clock) times of the accelerated CCG and CCG: epri-21 (left) and uiuc-150 (right).}
    \label{fig:computation_pentagon}
\end{figure}

\subsection{Performances of the distributionally robust transmission grid operations} \label{Performance_comparison}
In this section we show performances of AC and GIC feasible distributionally robust transmission grid operations produced by our approach compared with the transmission grid operations obtained by two other approaches.\\

\noindent
\textbf{Approach 1:} 
This is the proposed AC and GIC feasibility approach for the DRO model in this paper. Obtain switching decision $\hat{\zb}^{\text{\tiny A1}}$ for the relaxed DRO model \eqref{DROLB-Simple} under pentagon support set.
Solve problem \eqref{DRO-Simple:Reform1} to obtain feasible $\yb(\hat{\zb}^{\text{\tiny A1}})$. \\ 

 \noindent
 \textbf{Approach 2:} This is for only the mean value of uncertain E-fields (akin to deterministic version without uncertainty in \cite{lu2017optimal}). Fixing $\tilxi = \mub$, the relaxed DRO model \eqref{DROLB-Simple} reduces to
\begin{small}
\begin{align}
	\min_{(\zb, \yb) \in \mathcal{S}} \ & \qb^{\text{\tiny T}} \zb + \bb^{\text{\tiny T}} \yb  + \mathcal{Q} (\zb, \yb, \mub).       
	\label{Deterministic-Simple}
\end{align}
\end{small}
Solve problem \eqref{Deterministic-Simple} to obtain switching decision $\hat{\zb}^{\text{\tiny A2}}$, and solve problem \eqref{DRO-Simple:Reform1} to obtain feasible $\yb(\hat{\zb}^{\text{\tiny A2}})$. \\ 

\noindent
\textbf{Approach 3:} This is without performing any mitigation actions. Set $\hat{\zb}^{\text{\tiny A3}} = 1$, in other words, switching on all generators, transmission lines, and transformers, and solve problem \eqref{DRO-Simple:Reform1} to obtain feasible $\yb(\hat{\zb}^{\text{\tiny A3}})$. \\

\noindent
\textbf{Sampling for Problem \eqref{DRO-Simple:Reform1}: } As described in Section \ref{subsec:ACFeasible}, solving problem \eqref{DRO-Simple:Reform1} is challenging. Hence we utilize the restricted inner max problem \eqref{restricted_inner_max} with $\Xi^{\text{\tiny S}}$ to obtain an UB in Algorithm \ref{algo:ModifiedCCG}. As the cardinality of $\Xi^{\text{\tiny S}}$ gets larger, the obtained solution converges to an optimal solution of Problem \eqref{DRO-Simple:Reform1}. 
{\color{black} 
To see this numerically, we construct a sufficiently large support set $\widehat{\Xi}^{\text{\tiny S}}$ that is close to $\Xi$, and a smaller support set $\Xi^{\text{\tiny S}}$, i.e., $|\Xi^{\text{\tiny S}}| \leq |\widehat{\Xi}^{\text{\tiny S}}|$.
For given $\hat{\zb}$ and $\Xi^{\text{\tiny S}}$, we compute the AC and GIC feasible solution denoted by $\hat{\yb}$, $\hat{\lambdab}$, and $\hat{\eta}(\Xi^{\text{\tiny S}}) := \max_{\tilxi \in \Xi^{\text{\tiny S}}} \  \mathcal{V} (\hat{\zb}, \hat{\yb}, \tilxi) -  \hat{\lambdab}^{\text{\tiny T}} \tilxi$.
To check if $\Xi^{\text{\tiny S}}$ is appropriately chosen via sampling, we verify the following inequalities to hold true:
\begin{align}
\hat{\eta}(\Xi^{\text{\tiny S}}) \geq \mathcal{V} (\hat{\zb}, \hat{\yb}, \tilxi) - \hat{\lambdab}^{\text{\tiny T}} \tilxi, \ \forall \tilxi \in \widehat{\Xi}^{\text{\tiny S}}. \label{violation}
\end{align}
If the inequalities in \eqref{violation} are not violated, the solution $\textbf{g}(\Xi^{\text{\tiny S}}) := \big(\hat{\yb}, \hat{\lambdab}, \hat{\eta}(\Xi^{\text{\tiny S}}) \big)$ is also feasible to $\widehat{\Xi}^{\text{\tiny S}}$.
In Table \ref{tab:feasibility}, we display the number of violations on \eqref{violation} for various $\textbf{g}(\Xi^{\text{\tiny S}})$, which is computed by fixing $\hat{\zb} = \hat{\zb}^{\text{\tiny A1}}$, where $\hat{\zb}^{\text{\tiny A1}}$ is obtained by solving the epri-21 instance under the pentagon support set $\Xi$ with $R=10$, $M=0.5R$, and $\delta^{\mu}=45^{\circ}$.
Note that we uniformly sample $\tilxi$ from $\Xi$ to construct (i) $\widehat{\Xi}^{\text{\tiny S}}$ such that $|\widehat{\Xi}^{\text{\tiny S}}|=5,000$, and  (ii) $\Xi^{\text{\tiny S}}$ such that $|\Xi^{\text{\tiny S}}| \in \{1, 5, 10, 25, 50\}$. 
The results in Table \ref{tab:feasibility} show that the solution $\textbf{g}(\Xi^{\text{\tiny S}})$ with $|\Xi^{\text{\tiny S}}|=50$ is feasible to $\widehat{\Xi}^{\text{\tiny S}}$.
We observe that the similar results also hold for the other instances.
\begin{table}[h!]
    \centering
    \caption{Number of inequalities \eqref{violation} violated by $\textbf{g}(\Xi^{\text{\tiny S}})$ under various $\Xi^{\text{\tiny S}}$, when $|\widehat{\Xi}^{\text{\tiny S}}|=5,000$.}
    \begin{tabular}{|c|c|c|c|c|c|}
    \hline
    $|\Xi^{\text{\tiny S}}|$ & 1 & 5 & 10 & 25 & 50  \\ \hline
    Number of violations & 5000 & 1976 & 934 & 253 & 0 \\ \hline
    \end{tabular}
    \label{tab:feasibility}
\end{table}

}

To summarize the importance of the above-mentioned approaches, comparing Approaches 2 and 3 helps us understand the importance of controlling switching components in the context of mitigating GIC effects. 
Also, by comparing Approaches 1 and 2, one can understand the importance of distributionally robust transmission grid operations under GMD events with uncertain E-fields.
In the next two subsections, we compare these three approaches using epri-21 and uiuc-150 systems.

\subsubsection{The epri-21 system} \label{epri-21}
We focus on instances with $R \in \{2.5, 5, 7.5, 10, 12.5 \}$, $M=0.5R$, and $\delta^{\mu} = 45^{\circ}$.

First, in {\color{black} Table \ref{Table:Solution_epri21}} 
we show switching decisions $\hat{\zb}^{\text{\tiny A1}}$ and $\hat{\zb}^{\text{\tiny A2}}$ obtained by solving problems \eqref{DROLB-Simple} and \eqref{Deterministic-Simple}, respectively. 
Since Approach 1 hedges against the uncertain GMD events, the number of mitigation actions (via switching off network components) is larger in comparison with {\color{black}deterministic ones, i.e., Approach 2}.

\begin{table}[ht]
\centering
\caption{Switched off components for epri-21 system. All generators are activated for both Approaches 1 and 2.}
\begin{threeparttable}
\begin{tabular}{c|cc|cc}
\hline
 	& \multicolumn{2}{c|}{Approach 1} & \multicolumn{2}{c}{Approach 2} \\ 
 $R$ & Line & Transformer & Line & Transformer\\ \hline
$2.5$ & 3 & 18, 23, 29 & 2, 11 & 22 \\
$5$  & 3 & 18, 20, 23, 28 & 3 & 18, 21, 22, 28  \\
$7.5$  & 11 & 18, 21, 23, 28 & - & 28  \\
$10$  & 2, 8 & 22, 28 & 3, 8 & 23 \\
$12.5$  & 3 & 18, 19, 20, 29 & - & 21, 23 \\ \hline
\end{tabular}
\end{threeparttable}
\label{Table:Solution_epri21}
\end{table}

Next, we compare the three different switching decisions $\hat{\zb}^{\text{\tiny A1}}$, $\hat{\zb}^{\text{\tiny A2}}$, and $\hat{\zb}^{\text{\tiny A3}}$ by solving problem \eqref{DRO-Simple:Reform1} with respect to (i) costs in {\color{black}Table \ref{Table:Costs_epri21}}
and (ii) the total amounts of power loss (i.e., $\sum_{i \in \mathcal{N}}\sqrt{(\lpm_i)^2+(\lqm_i)^2}$), power load shed (i.e., $\sum_{i \in \mathcal{N}}\sqrt{(\lpp_i)^2+(\lqp_i)^2}$), and the additional reactive power loss due to GICs (i.e., $\sum_{i \in \mathcal{N}}d^{\text{\tiny qloss}}_i$) {\color{black}in Figure \ref{fig:epri21_LossComparison}}.

\begin{table}[ht]
	\centering
	\caption{Cost components for epri-21 system.}
	\begin{threeparttable}	
	\begin{tabular}{r|r|r|r|r|r|r}
	\hline
		& & $R=2.5$ & $R=5$ & $R=7.5$ & $R=10$ & $R=12.5$ \\ \hline	  
	\multirow{ 4}{*}{A1} & $C_1$ & 398.2K   & 398.2K & 398.2K & 398.2K & 398.3K  \\ 
	& $C_2$ & 0.1  & 0.0 & 0.1 & 0.1 & 0.0 \\ 
	& $C_3$ & 0.2  & 0.0 & 0.2 & 0.2 & 0.0 \\ 
	& $C_4$ & 398.2K  & 398.2K & 398.2K& 398.2K & 398.3K \\ \hline
	\multirow{ 4}{*}{A2}& $C_1$ & 398.2K   & 399.2K & 399.1K & 399.9K & 400.1K  \\ 
	& $C_2$& 0.1  & 148.6K & 596.2K & 800.1K & 1552.9K \\ 
	& $C_3$& 0.2  & 28.2K & 113.3K & 120.0K & 264.4K \\ 
	& $C_4$& 398.2K  & 576.0K & 1108.6K & 1320.0K & 2217.4K \\ \hline
	\multirow{ 4}{*}{A3}& $C_1$ & 398.2K   & 399.2K & 399.4K & 399.7K & 400.2K  \\ 
	& $C_2$& 0.0  & 343.2K & 1091.9K & 1825.8K & 2549.2K \\ 
	& $C_3$& 0.0  & 90.7K & 106.7K & 141.6K & 176.6K \\ 
	& $C_4$& 398.2K  & 833.1K & 1598.0K & 2367.1K &3126.0K \\ \hline
	\end{tabular}	
	\begin{tablenotes}\scriptsize    
	\item[] $C_1$: generation cost; $C_2$, penalty for load shedding, $C_3$, worst-case expected damage due to GICs, $C_4:$ total costs.
	\end{tablenotes}
	\end{threeparttable}
	\label{Table:Costs_epri21}
\end{table}

\begin{figure}[!h]
    \centering
    \includegraphics[scale=0.22]{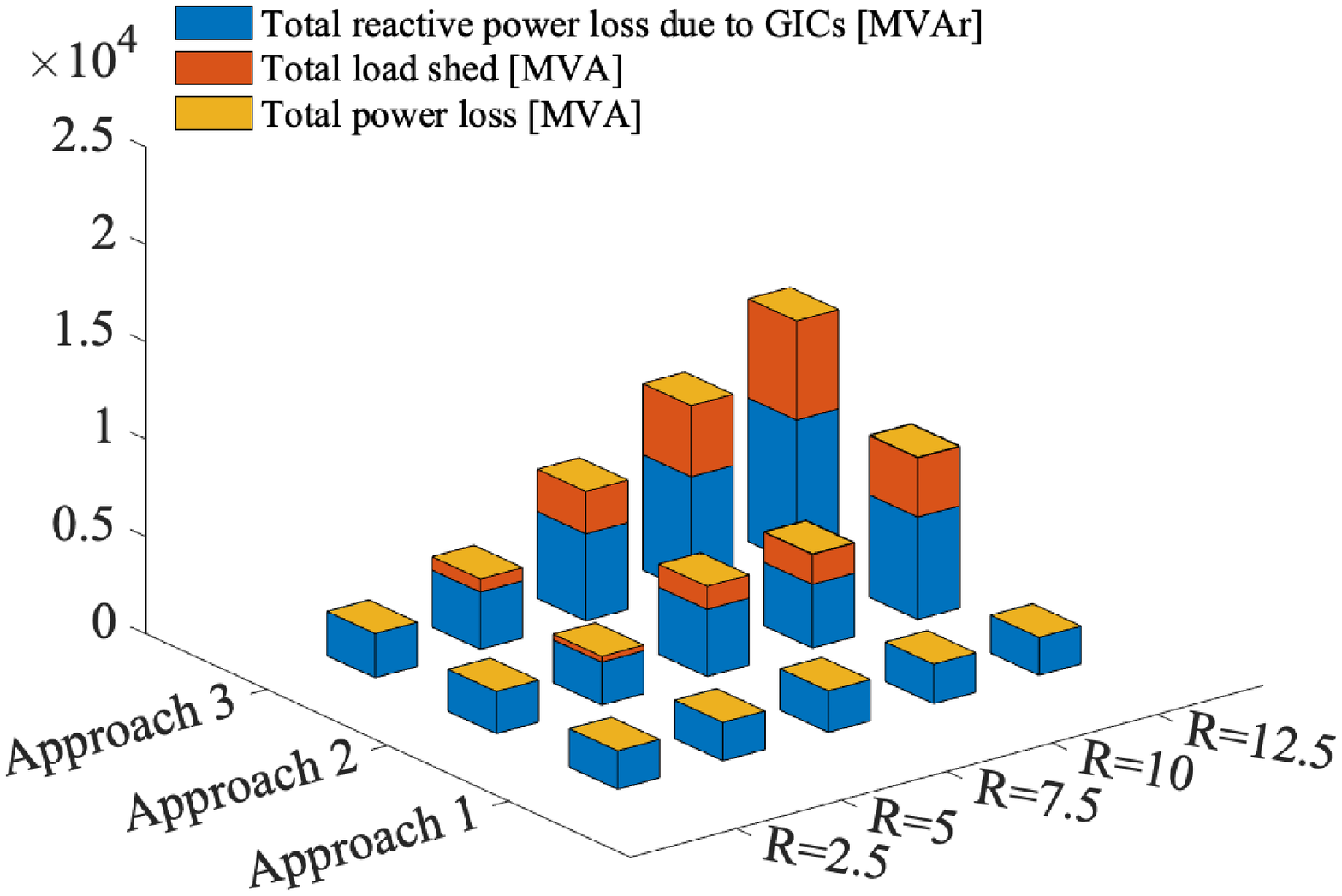}
    \includegraphics[scale=0.22]{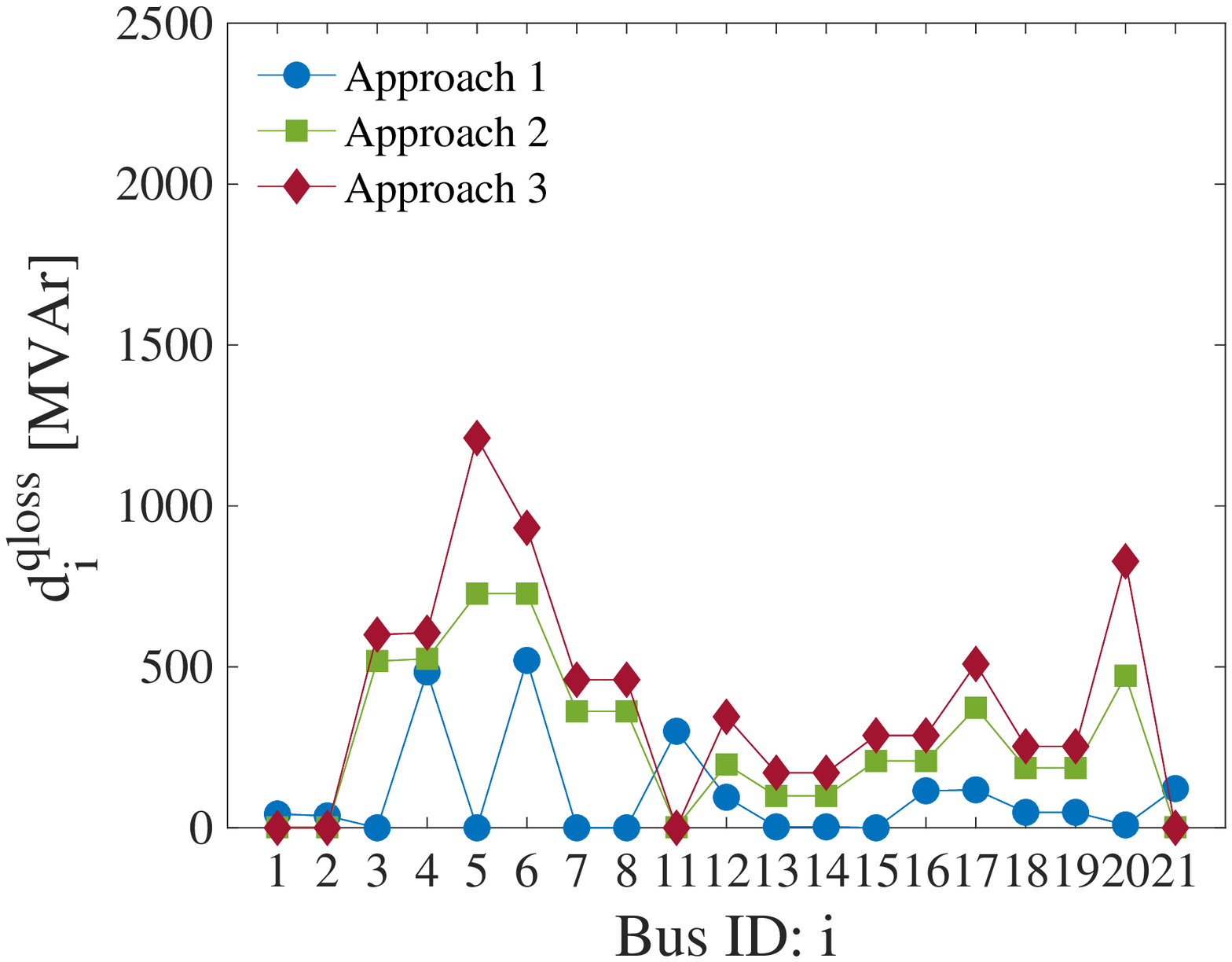}
    \caption{Comparison of three approaches based on total reactive power loss due to GICs, power loss, and load shed (left), and $d^{\text{\tiny qloss}}$ for each bus when $R=12.5$ (right).}
    \label{fig:epri21_LossComparison}
\end{figure}

In {\color{black} Table \ref{Table:Costs_epri21}},
as expected, we observe that the total costs obtained by Approach 3 dramatically increase as $R$ increases.
With $(\hat{\zb}^{\text{\tiny A2}}, \yb(\hat{\zb}^{\text{\tiny A2}}))$, we can reduce the total costs, but still need to pay some penalty costs {\color{black}(see $C_2$ and $C_3$ for A2 in Table \ref{Table:Costs_epri21})}. 
Fortunately, with the distributionally robust transmission grid operations $(\hat{\zb}^{\text{\tiny A1}}, \yb(\hat{\zb}^{\text{\tiny A1}}))$ obtained by Approach 1, we can avoid paying not only the penalty for load shedding, but also the worst-case expected damage due to GICs.
This result implies that the switching decisions obtained by the relaxed DRO model \eqref{DROLB-Simple} can lead to reliable and feasible grid operations.

In Figure \ref{fig:epri21_LossComparison} (left), we observe that the total load shed and the additional reactive power losses due to GICs ($d^{\text{\tiny qloss}}$) obtained by Approach 3 increase, as $R$ increases.
Without any mitigation actions, with all lines and transformers switched on, $d^{\text{\tiny qloss}}$ increases as $R$ increases, and thus the total load shed also increases in the system.
This negative GIC effect can be mitigated by controlling the switching components, as in Approaches 1 and 2.
With Approach 1, however, we can further reduce reactive power losses in the system.
For example, when $R=12.5$V/km, $d^{\text{\tiny qloss}}$ obtained by Approach 1 is lower than $d^{\text{\tiny qloss}}$ obtained by Approaches 2 and 3, as depicted in Figure \ref{fig:epri21_LossComparison} (right).
This observation explains why our proposed methods in Approach 1 leads to lower costs and a more reliable grid w.r.t \textit{significantly smaller} load sheds and reactive power losses with related voltage instability issues, compared with those obtained by Approaches 2 and 3.

\subsubsection{The uiuc-150 system} \label{uiuc-150}
We focus on instances with $R \in \{2.5, 5, 7.5, 10, 12.5 \}$, $M=0.5R$, and $\delta^{\mu} = 135^{\circ}$.

{\color{black}Table \ref{Table:Solution_uiuc150}} shows switching decisions ($\hat{\zb}^{\text{\tiny A1}}$, $\hat{\zb}^{\text{\tiny A2}}$) and their total costs obtained by solving problems \eqref{DROLB-Simple} and \eqref{Deterministic-Simple}, respectively.
Compared with $\hat{\zb}^{\text{\tiny A2}}$, the switching decision $\hat{\zb}^{\text{\tiny A1}}$ in Approach 1 suggests switching off more transmission lines as $R$ increases, in order to hedge against the uncertainty. 

Next, we compare the performances of switching decisions ($\hat{\zb}^{\text{\tiny A1}}$, $\hat{\zb}^{\text{\tiny A2}}$, and $\hat{\zb}^{\text{\tiny A3}}$) by solving problem \eqref{DRO-Simple:Reform1} as described in Section \ref{Performance_comparison}.
{\color{black}In Table \ref{Table:Costs_uiuc150}, we compare three approaches with respect to costs.}
When $R \in \{2.5, 5, 7.5\}$, the total costs obtained by the three approaches are close because of the similarity of $\hat{\zb}^{\text{\tiny A2}}$, $\hat{\zb}^{\text{\tiny A2}}$, and $\hat{\zb}^{\text{\tiny A3}}$.
In contrast, when $R \in \{10, 12.5\}$ (larger E-field magnitudes), we observe that $\hat{\zb}^1$ can effectively reduce the total cost.
Especially when $R=12.5$V/km, $\hat{\zb}^{\text{\tiny A1}}$ can prevent shedding power load and achieve reliable operations by switching off a few transmission lines and transformers.
While Approach 3 sheds a significant amount of power load in advance to prepare for future severe GMD events due to lack of mitigation actions, Approach 2 sheds relatively less power load since $\hat{\zb}^{\text{\tiny A2}}$ is a switching decision that prepares for a specific GMD event, namely, when $\tilxi=\mub$ (mean value). 

\begin{table}[!h]
	\centering
	\caption{Switched off components for uiuc-150 system. All generators are activated for both Approaches 1 and 2.}
	\begin{threeparttable}	
	\begin{tabular}{c|cc|cc}
	\hline
		 & \multicolumn{2}{c|}{Approach 1} & \multicolumn{2}{c}{Approach 2} \\ 
	 $R$ & Line & Transformer & Line & Transformer\\ \hline
	$2.5$ & 122 & - & 4,84,109 & - \\  \hline
	$5$  & \makecell{22,50,92,122,124,146} & - & 23,42,92 & - \\ \hline
	$7.5$  & 55,86 & 180 & 59,92 & -  \\ \hline
	$10$  & \makecell{24,28,50,53,85,\\109,122,124,155} & - & \makecell{13,28,\\67,86,92}& - \\ \hline
	$12.5$  & \makecell{22,39,43,44,50,\\59,67,81,121,122,\\128,129,150,151} & 187 & \makecell{28,59,\\67,79,92} & 187,215 \\ \hline
	\end{tabular}
	\end{threeparttable}
	\label{Table:Solution_uiuc150}
\end{table}

\begin{table}[ht]
	\centering
	\caption{Cost components for uiuc-150 system.}
	\begin{threeparttable}	
	\begin{tabular}{r|r|r|r|r|r|r}
	\hline
		 & & $R=2.5$ & $R=5$ & $R=7.5$ & $R=10$ & $R=12.5$ \\ \hline
	\multirow{ 4}{*}{A1} & $C_1$ & 801.2K   & 812.7K & 898.9K & 805.4K & 784.2K  \\ 
	 & $C_2$& 0.8  & 8.7K & 17.6K & 50.2K & 167.7K \\ 
	 & $C_3$& 1.3  & 1.3 & 0.0 & 110.0K & 1002.3K \\ 
	 & $C_4$& 801.2K  & 821.4K & 916.5K & 965.6K & 1954.2K \\ \hline
	\multirow{ 4}{*}{A2} & $C_1$ & 797.0K   & 830.1K & 927.5K & 775.2K & 643.8K  \\ 			
	 & $C_2$& 0.8  & 0.0 & 6.0K & 118.4K & 3666.6K \\ 
	 & $C_3$& 1.3  & 0.0 & 1.3 & 220.6K & 1221.9K \\ 
	 & $C_4$& 797.0K & 830.1K & 933.5K & 1114.2K & 5532.3K \\ \hline 
	\multirow{ 4}{*}{A3} & $C_1$ & 803.5K   & 826.3K & 909.3K & 772.4K & 484.5K  \\ 		
	 & $C_2$& 0.8  & 0.8 & 12.9K & 123.1K & 14239.0K \\ 
	 & $C_3$& 1.3  & 1.3 & 1.3 & 210.3K & 656.5K \\ 
	 & $C_4$& 803.5K  & 826.3K & 922.2K & 1105.8K & 15380.0K \\ \hline
	\end{tabular}		
	\begin{tablenotes}\scriptsize    
	\item[] $C_1$: generation cost; $C_2$, penalty for load shedding, $C_3$, worst-case expected damage due to GICs, $C_4:$ total costs.
	\end{tablenotes}
	\end{threeparttable}
	\label{Table:Costs_uiuc150}
\end{table}

{\color{black}
\subsection{DRO versus Stochastic Programming approach}
In this section we compare the DRO approach (i.e., Approach 1) with a stochastic programming (SP) approach via the out-of-sample simulation.
To this end, we consider the epri-21 system and generate $N^{\text{train}}=100$ training samples $\{ \widehat{\bm{\xi}}^{\ell} =(\widehat{M}^{\ell} \cos \widehat{\delta}^{\ell}, \widehat{M}^{\ell} \sin \widehat{\delta}^{\ell}) \}_{\ell=1}^{N^{\text{train}}}$, where $\widehat{M}^{\ell}$ and $\widehat{\delta}^{\ell}$ are uniformly extracted from $[0,10]$V/km and $[0^{\circ},180^{\circ}]$, respectively. Similarly, we generate $N^{\text{test}}=1000$ testing samples for the out-of-sample simulation.
\subsubsection{Switching decisions}
First, we use the relaxed DRO model \eqref{DROLB-Simple} to find a DRO switching decision $\widehat{z}^{\text{A1}}$. 
The pentagon support set described in Section \ref{sec:E-field} is constructed with the following parameters: 
$R=10$,
$\mu_{\text{\tiny E}} = \widehat{M}^{\text{avg}} \cos \widehat{\delta}^{\text{avg}}$, and
$\mu_{\text{\tiny N}} = \widehat{M}^{\text{avg}} \sin \widehat{\delta}^{\text{avg}}$, 
where $\widehat{M}^{\text{avg}}$ and $\widehat{\delta}^{\text{avg}}$ are the average of the training samples.
Second, we solve the following SP model to find a SP switching decision $\widehat{z}^{\text{SP}}$:

\vspace{-2mm}
\begin{small}
	\begin{align}
	\min_{(\zb, \yb) \in \mathcal{S}} \ & \qb^{\text{\tiny T}} \zb + \bb^{\text{\tiny T}} \yb  + \frac{1}{N^{\text{train}}} \sum_{\ell=1}^{N^{\text{train}}} \mathcal{Q} (\zb, \yb, \widehat{\bm{\xi}}^{\ell}).
	\label{SP-Simple}
	\end{align}
\end{small}	
We present switching decisions obtained from both DRO and SP approaches in Table \ref{Table:Solution_DRO_SP}. 

\begin{table}[ht]
	\centering
	\caption{{\color{black} Switched off components for the epri-21 system while comparing DRO vs. SP approaches. All generators are activated for both approaches.}}
	\begin{threeparttable}	
	\begin{tabular}{c|cc|cc}
	\hline
		 & \multicolumn{2}{c|}{DRO} & \multicolumn{2}{c}{SP} \\ 
	 & Line & Transformer & Line & Transformer\\ \hline	
	epri-21 & - & 21,22 & 11 & 29 \\ \hline	
	\end{tabular}
	\end{threeparttable}
	\label{Table:Solution_DRO_SP}
\end{table}

\subsubsection{AC power flow solution}
Recall that both \eqref{DROLB-Simple} and \eqref{SP-Simple} are MISOCP models because the feasible region $\mathcal{S}$ contains SOC constraints and binary variables (i.e., switching decisions).
We modify the models by replacing $\mathcal{S}$ with the nonconvex feasible region of $\mathcal{A}$ from \eqref{DRO-Simple} and fixing the switching decisions $\widehat{z}^{\text{A1}}$ and $\widehat{z}^{\text{SP}}$ obtained from the relaxed models to obtain AC power flow solutions (i.e., $\widehat{y}^{\text{A1-AC}}$ and $\widehat{y}^{\text{SP-AC}}$, respectively).
Note that the values of slack variables ($ \lpp, \lqp, \lpm, \lpp$) measure the AC feasibility. We have indeed verified that both $\widehat{y}^{\text{A1-AC}}$ and
$\widehat{y}^{\text{SP-AC}}$ are AC feasible.

\subsubsection{Out-of-sample simulations}
We aim to compare the out-of-sample performance of the solution $(\widehat{z}^{\phi},\widehat{y}^{\text{$\phi$-AC}})$, where $\phi \in \{\text{A1}, \text{A3}, \text{SP}\}$. Note that the solution obtained by Approach 3 is also considered to emphasize the importance of making switching decisions via the DRO and SP approaches.
To this end, we compute the average damage cost due to GICs using the out-of-sample scenarios $N^{\text{test}}$, namely, 
\begin{align}
& O^{\phi}:= \frac{1}{N^{\text{test}}} \sum_{\ell = 1}^{N^{\text{test}}} \mathcal{V}(\widehat{z}^{\phi},\widehat{y}^{\text{$\phi$-AC}} , \widehat{\bm{\xi}}^{\ell} )  	\nonumber
\end{align}
where $\mathcal{V}$ is defined in \eqref{DRO-Simple}. 
With $N^{\text{test}}=1000$, we obtain $O^{\text{A1}}=114.68$ and $O^{\text{SP}}=600.57$, which are much smaller than $O^{\text{A3}}=12004.23$ (i.e., the average damage cost due to GICs when all components are switched on).
This result demonstrates that making switching decision via the DRO approach can significantly reduce the average damage cost due to GICs under the out-of-sample scenarios.
}

\section{Conclusions} \label{sec:conclusion}
In this paper, we proposed a DRO model that finds nonlinear AC and GIC feasible transmission grid operations that mitigate potential negative impacts of uncertain E-fields due to GMDs. The proposed DRO model is solvable by using the standard CCG algorithm on small-scale grids. To improve the computational performance of CCG, we proposed an \textit{accelerated} CCG that takes advantage of the MISOCP reformulation of the DRO model under the triangle support set.
We numerically showed the run-time efficacy of this reformulation. Furthermore, we provided a detailed case study on the epri-21 and uiuc-150 systems that analyzed the effects of modeling uncertain GMDs.
Compared with the other existing stochastic programming approaches, out-of-sample testing suggests that our DRO approach can provide cheaper feasible and reliable transmission grid operations that reduce the amounts of load shedding and reactive power loss due to GICs, thus saving significant amounts of operating costs. 

{\color{black} An assumption in this paper was that the induced E-field due to the GMD event was spatially uniform within the entire electric grid. However, considering the spatial scales for large-scale grids and localized magnetic-storm disturbances, we plan to generalize the proposed two-stage DRO approach to account for non-uniform, uncertain E-fields on the GIC-induced voltages sources in out future work.} \\  

\noindent
\textbf{Acknowledgments} This work was supported by the U.S. Department of Energy's LDRD programs at Los Alamos National Laboratory under (a) \emph{``20170047DR: Impacts of Extreme Space Weather Events on Power Grid Infrastructure: Physics-Based Modelling of Geomagnetically-Induced Currents (GICs) During Carrington-Class Geomagnetic Storms"} and (b) \emph{``LDRD 20190590ECR: Discrete Optimization Algorithms for Provably Optimal Quantum Circuit Design"}.
The material is also based in part upon work supported by the U.S Department of Energy, Office of Science, under contract DE-AC-06CH11357.

\bibliographystyle{IEEEtran}
\bibliography{references.bib}

\end{document}